%% file: dirichlet.tex
\documentclass[a4paper,11pt, reqno]{amsart}
\usepackage[DIV=12, oneside]{typearea}
\usepackage[utf8]{inputenc}
\usepackage[T1]{fontenc}
\usepackage[english]{babel}
\usepackage{amsmath}
\usepackage{esint}
\usepackage{amssymb, amsthm, bbm, dsfont}
\usepackage{enumerate}
\usepackage{color, mathrsfs}
\usepackage{thmtools}
\usepackage{todonotes}
\usepackage[missing={noch nicht eingerichtet}]{gitinfo}
\usepackage{eso-pic}
\usepackage{isorot}
\usepackage[shortlabels]{enumitem}
\usepackage{microtype}
\usepackage[colorlinks=true, urlcolor=blue, linkcolor=blue,
citecolor=black]{hyperref}
\usepackage{array}
\usepackage{booktabs}
\usepackage{stmaryrd}
\newcommand{\njet}{$\pmb{-}$}
\theoremstyle{plain}
\declaretheorem[title=Theorem, parent=section]{thm}
\declaretheorem[title=Lemma,sibling=thm]{lem}
\declaretheorem[title=Corollary,sibling=thm]{cor}
\declaretheorem[title=Proposition,sibling=thm]{prop}

\declaretheoremstyle[bodyfont=\normalfont, qed=$\clubsuit$]{endsign}
\declaretheorem[title=Remark, sibling=thm, style=endsign]{bem}
\theoremstyle{definition}
\declaretheorem[title=Definition,sibling=thm]{defi}
\declaretheorem[title=Example]{example}
\numberwithin{equation}{section}
\input{kurzkommandos-dirichlet}

\newlength{\widestlabel}
\setlist{itemindent=\parindent}
\setlist[enumerate]{labelwidth=\widestlabel, leftmargin=!}
\addto\extrasenglish{%

}
\author{Matthieu Felsinger}
\author{Moritz Kassmann} 
\author{Paul Voigt}
\address{Fakultät für Mathematik\\
Universität Bielefeld \\
Postfach 100131 \\
D-33501 Bielefeld}
\email{m.felsinger@math.uni-bielefeld.de}
\address{Fakultät für Mathematik\\
Universität Bielefeld \\
Postfach 100131 \\
D-33501 Bielefeld}
\email{moritz.kassmann@uni-bielefeld.de}
\address{Fakultät für Mathematik\\
Universität Bielefeld \\
Postfach 100131 \\
D-33501 Bielefeld}
\email{pvoigt@math.uni-bielefeld.de}
\title{The Dirichlet problem for nonlocal operators}
\thanks{We thank Tadele Mengesha for several helpful comments.}

\begin{document}
%%%%
%\PrintVersionNo
%%%%

\begin{abstract}
In this note we set up the elliptic and the parabolic Dirichlet problem for
linear nonlocal operators.  As opposed to the classical case of second order
differential operators, here the ``boundary data'' are prescribed on the
complement of a given bounded set. We formulate the problem in the classical
framework of Hilbert spaces and prove unique solvability using standard
techniques like the Fredholm alternative.
\end{abstract}
\maketitle
%%%%%%%
\section{Introduction}\label{sec:intro}

Given an open and bounded set $\Omega\subset\R^d$ and functions $f: \Omega \to
\R$ and $g:\partial \Omega \to \R$, the classical Dirichlet problem is to find a
function $u:\Omega \to \R$ such that
\begin{subequations}
\label{Classical_Dirichlet_Problem}
\begin{align}
 -\Delta u &= f \qquad \text{ in } \Omega \,,
\label{eq:D-problem_classical_eq}\\
  u &= g \qquad \text{ on } \partial \Omega \,.
\label{eq:D-problem_classical_bd}
\end{align}
\end{subequations}
A more general problem is to solve, instead of
equation \eqref{eq:D-problem_classical_eq}, the equation
\begin{align}
 -\operatorname{div} \left( A(\cdot) \nabla u \right) &= f \qquad \text{
in } \Omega \,,
\label{eq:D-problem_div_form_eq}
\end{align}
where $A(x)=(a_{ij}(x))$ is
a $d\times d$-matrix which is uniformly positive definite and uniformly
bounded. When working in Hilbert spaces a standard assumption is $f
\in H^{-1}(\Omega)=(H_0^1(\Omega))^*$ and
$g \in H^1(\Omega)$. $H^1(\Omega)$ denotes the Hilbert space of all
$L^2(\Omega)$-functions
with a distributional derivate in $L^2(\Omega)$ and $H^1_0(\Omega)$
is the closure of $C^\infty_c(\Omega)$ with respect to the norm of
$H^1(\Omega)$. The Riesz representation Theorem resp. the Lax-Milgram Lemma
imply that there is a unique function $u \in H^1(\Omega)$ such that $u-g \in
H^1_0(\Omega)$ and for every $v \in H^1_0(\Omega)$ 
\begin{align}\label{eq:D-problem_div_form_weak_eq}
 \left( A(\cdot) \nabla u ,\nabla v \right)_{L^2(\Omega)} = \ska{f,v}\,. 
\end{align}
This is one way to solve the Dirichlet problem. Note that the bilinear form on
the left-hand-side of \eqref{eq:D-problem_div_form_weak_eq} is not necessarily
symmetric. An important further extension is given when one considers
additional terms of lower order in \eqref{eq:D-problem_div_form_eq}. 

The aim of this article is to provide a similar set-up for the Dirichlet
problem associated with nonlocal operators. Our main objects are nonlocal
operators $\opL$ of the form
\begin{align}\label{eq:nonlocal-op}
\left(\opL u\right)(x)=\lim\limits_{\varepsilon \to
0+} \int\limits_{y \in \R^d \setminus
B_\varepsilon(x)}(u(x)-u(y))k(x,y)\d y \,, 
\end{align}
where $k:\R^d \times \R^d \to [0,\infty]$ is measurable. Our focus is on kernels
$k$ with the following properties: (i) $k$ is not necessarily symmetric, (ii) $k$ might
be singular on the diagonal and\linebreak[2] (iii) $k$ is allowed to be discontinuous. These properties
imply
that $\cL u$ cannot be evaluated in the classical sense at a given point $x
\in \R^d$
even if $u$ is smooth.

Assume $\Omega\subset\R^d$ is open, bounded and $f \in L^2(\Omega)$. Given
a function $g:\complement\,\Omega  \to \R$, we want to find a solution $u:\R^d
 \to \R$ to the  problem 
\begin{subequations}
\label{Dirichletproblem_einleitung}
\begin{align} \label{eq:dirichlet-operatorform}
 {\opL}u & =f \quad\text{ in } \Omega \\
 u & =g \quad\text{ on } \complement \Omega.
\end{align}
\end{subequations}
In order to do so, we need to answer the following questions: Which is a
natural space for the data $g$ and for the solution $u$? In which sense is
equation \eqref{eq:dirichlet-operatorform} to be interpreted? Note that, opposed
to
the classical case of second order differential operators, here the data $g$ are
prescribed on the complement of a given bounded set. The term ``complement
value problem'' would be more appropriate for
\eqref{Dirichletproblem_einleitung} than ``boundary value problem''. 

\begin{bem}
It makes sense to study -- instead of \eqref{eq:nonlocal-op} -- nonlocal
operators of the form
\begin{align*}
\left({\opL}u\right)(x)=\lim\limits_{\varepsilon \to
0+} \int\limits_{y \in \Omega \setminus
B_\varepsilon(x)}(u(x)-u(y))k(x,y)\d y \,. 
\end{align*}
In that case it would not make sense to describe data on $\complement\, \Omega$.
For a large class of kernels one could study associated boundary value problems
with boundary data $g:\partial \Omega \to \R$. Note that such
operators appear as generators of so-called censored processes. We do
not study these cases in this article.
\end{bem}
%%%%%%
In this work we establish a Hilbert space approach to the nonlocal Dirichlet
problem \eqref{Dirichletproblem_einleitung}. We prove results on existence and
uniqueness of solutions under various assumptions on the kernels
$k:\R^d\times\R^d\to [0,\infty]$ which we assume to be measurable.  The
symmetric and anti-symmetric parts of $k$ are defined by 
\[k_s(x,y)=\frac{1}{2}\left(k(x,y)+k(y,x)\right)\quad \text{and}\quad
k_a(x,y)=\frac{1}{2}\left(k(x,y)-k(y,x)\right).\]
Clearly, $|k_a(x,y)|\leq k_s(x,y)$ for almost all $x,y\in\R^d$. Throughout
this article we assume that
the symmetric part of the kernel satisfies the following integrability
condition: 
\begin{equation}\label{eq:integrierbarkeitsbedingung-k_s}\tag{L}
 x\mapsto\int\limits_{\R^d}\left(1 \wedge \bet{x-y}^2\right)k_s(x,y)\d y 
\quad \in
L^1_{loc}(\R^d).
\end{equation}
We call a kernel $k$ {\bf integrable} if, for every
$x\in \R^d$ the quantity $\int_{\R^d} k_s(x,y) \d y$ is finite and the  
mapping $x\mapsto \int_{\R^d} k_s(x,y) \d y$ is locally integrable. We
call a kernel {\bf non-integrable} if it is not integrable in the sense above.
In this work we deal with integrable and non-integrable kernels at the same
time. In \autoref{sec:examples} we provide several characteristic
examples. A simple integrable example is given by
$k(x,y)=\mathds{1}_{B_1}(x-y)$, a simple non-integrable example by
$k(x,y)=|x-y|^{-d-1}$, both kernels being symmetric.

For a certain class of
kernels $k$ the Lax-Milgram Lemma turns out to be a good tool whereas for
another class the Fredholm alternative is more appropriate. Rather than going
through all assumptions in detail, for this introduction let us restrict
ourselves to an example in the simple setting where $\Omega$
equals the unit ball $B_1 \subset \R^d$. 

\begin{example}
Assume $0 < \beta <
\frac{\alpha}{2} < 1$. Let $I_1, I_2$ be arbitrary nonempty open subsets of
$S^{d-1}$ with $I_1=-I_1$. Set
$\mathcal{C}_j=\{h \in \R^d
|\, \tfrac{h}{|h|} \in I_j\}$ for $j \in \{1,2\}$ and  
\[k(x,y)=
|x-y|^{-d-\alpha}\mathds{1}_{\cC_1}(x-y)
+|x-y|^{-d-\beta}\mathds{1}_{\cC_2}(x-y) \mathds{1}_{B_1}(x-y) \,. \]
The part involving $|x-y|^{-d-\beta}$ can be seen as a lower order
perturbation of the main part of the kernel resp. integro-differential
operator produced by $|x-y|^{-d-\alpha}\mathds{1}_{\cC_1}(x-y)$. We discuss this
example in detail in \autoref{sec:examples}. Define complement data
$g:\R^d\to\R$ by
\[g(x)=\begin{cases}
        (|x|-1)^\gamma\qquad &\text{if }1\leq |x|\leq 2\,,\\
        0, &\text{else.}
       \end{cases}\]
where $\gamma$ is an arbitrary real number satisfying
$\gamma>\frac{\alpha-1}{2}$. Note that $g$ may be unbounded. Let $f \in
L^2(\Omega)$ be arbitrary. For such a kernel $k$ and such data $g$ and $f$ our
results imply that the Dirichlet problem \eqref{Dirichletproblem_einleitung}
for $\Omega=B_1 \in \R^d$ has a unique variational solution $u$.

Another interesting example to which our theory applies, is given by $k(x,y) =
b(x) |x-y|^{-d-\alpha(x)}$ under certain assumptions on the functions $b:\R^d
\to (0,\infty)$ and $\alpha:\R^d \to (0,2)$, see Example \ref{exa:schilling} in
\autoref{sec:examples}.

In \autoref{cor:wellposedness-para-nonzero} we provide an existence result for parabolic equations
(see \eqref{eq:dirichlet-para-nonzero}) where the operator contains a time-dependent kernel $k_t$ 
and where the complement data and the initial value are prescribed by a single function $h(t,x)$. 
This result covers the case $k_t=k$ and $h(t,x)=g(x)$ with $g$ and $k$ as above.
\end{example}
%%%%%%
Let us comment on related results in the literature. It is remarkable that,
although the questions and the framework of our studies are quite basic, the
results have not been established yet. One reason might be that
integro-differential operators not satisfying the transmission property
have been studied only recently when considered in bounded domains. On the
other hand, the proof of the weak maximum principle for nonsymmetric nonlocal
operators, \autoref{thm:weak-maximum}, turns out to be quite tricky.  

To our best knowledge, the first profound study of a nonlocal Dirichlet problem
for integro-differential operators violating the transmission property is
provided in \cite[Ch. VII]{BlHa86} using balayage spaces and the method
of Perron, Wiener and Brelot. A detailed study of nonlocal Dirichlet problems 
including a Hilbert space approach  can be found in \cite{Hoh_Jakob96}. Section
6 of \cite{Hoh_Jakob96} addresses the questions of this article but assumes the
bilinear forms to be symmetric. As the
proofs show, symmetry of $k$ simplifies the situation greatly. Symmetry is not
assumed in the previous sections of \cite{Hoh_Jakob96} but there, the symbol
(i.e. the generator) is used which would lead to strong assumptions on
the kernel $k$. Note that in our framework, typically the integro-differential
operator cannot be evaluated pointwise, even when applied to smooth functions. 

The Dirichlet problem for nonlocal operators is studied for fully nonlinear
problems in \cite{BCI08} using viscosity solutions. There, the complement data
are chosen independently from the kernels which is very different from our
approach where, for every $k$ there is an appropriate function space for the
data $g$. 

Variational solutions to nonlocal problems have been already
considered by groups working in the theory of peridynamics, e.g. in 
\cite{AksoyluMengesha2010, DGLZ12, BeMo13, DuMe13a, DuMe13b}. This theory
describes a nonlocal continuum model for
problems involving discontinuities or other singularities in the deformation.
The corresponding integration kernels are often integrable but non-integrable
kernels are considered too. What we call ``complement data'' in this article is
called ``volume constraints'' in articles on peridynamics. An important tool
in the articles above are appropriate Poincar\'{e}-Friedrichs
inequalities. It occurs that \autoref{lem:poincare-nonsingular}, in the case of
scalar functions, is more general than corresponding Poincar\'{e}-Friedrichs
inequalities in the aforementioned articles. A major difference of our approach
is that we study general (non-symmetric) kernels and non-zero
complement data together with appropriate function spaces.

We do not at all discuss regularity of solutions. Regularity up to the
boundary is studied carefully in \cite{ROSe13} for the case of the fractional
Laplacian. 

Results on nonsymmetric nonlocal Dirichlet forms are related to our
work, see our discussion of \eqref{eq:domination_by_ktilde} below. We refer
the reader to  \cite{HMS10}, \cite{MaSu12}, \cite{FuUe12}
and \cite{SchillingWang2011}. We believe that our examples can be helpful
when studying Hunt processes generated by nonsymmetric nonlocal
(semi-)Dirichlet forms.

The article is organized as follows. In \autoref{sec:setup} we define the
function spaces needed for our approach and explain their basic properties. We
also define what a solution of \eqref{Dirichletproblem_einleitung} shall be.
\autoref{subsec:poincare} is
dedicated to nonlocal versions of the Poincar\'{e}-Friedrichs inequality. In
\autoref{sec:lax_milgram} we prove a G{\aa}rding inequality, comment on the
sector condition and apply the
Lax-Milgram Lemma to the nonlocal
Dirichlet problem. \autoref{sec:uniqueness-fredholm} is devoted to the weak
maximum
principle for integro-differential operators in bounded domains. This tool
is applied when using the Fredholm alternative in \autoref{subsec:fredholm}. Our existence and uniqueness results are extended to the
parabolic Dirichlet
problem in \autoref{sec:parabolic}. Finally, in \autoref{sec:examples}
we provide
many detailed examples of kernels $k$ and discuss their properties. Although,
formally, this section could be omitted, it presumably is one of the
most important one for every reader.\\
A short index with all conditions used in this work is given at the end of \autoref{sec:examples} on p.~\pageref*{index}.

\section{Function spaces and variational solutions}\label{sec:setup}
In this section we derive a weak resp. variational formu\-lation of
\eqref{Dirichletproblem_einleitung}. We start with the defi\-nition of the
relevant function spaces.

\subsection{Function spaces}{\ }
%%%%%%%%%%%%%%
%%%%%%%%%%
\begin{defi}[Function spaces]\label{defi:function-spaces}
Let $\Omega\subset\R^d$ be open and assume that the kernel $k$ satisfies \eqref{eq:integrierbarkeitsbedingung-k_s}. We define the following linear spaces:
\begin{enumerate}[(i)]
\item $L^p_\Omega(\R^d)=\{u\in L^p(\R^d) \colon u=0 \text{ a.e. on
}\complement\,\Omega\}$.
\item Define 
\begin{align*}
&\VOmk =\Big\{v:\R^d \to \R \colon v|_\Omega
\in L^2(\Omega),
 \left(v(x)-v(y)\right)
k_s^{1/2}(x,y) \in L^2(\Omega \times \R^d) \Big\} \,,\\
&\big[u,v\big]_{V(\Omega;k)}=\iint\limits_{\Omega\,\R^d}\left[u(x)-u(y)\right]
\left[v(x)-v(y)\right] k_s(x,y)\d y\d x\,. 
\end{align*}
A seminorm on $\VOmk$ is given by $\left[v, v \right]_{\VOmk}$. Let us
emphasize that $\Omega = \R^d$ is allowed in this definition. In this case we
write $V(\R^d;k)=H(\R^d;k)$ and a norm -- which is obviously induced by a scalar
product -- on this space is defined by
\[\|v\|_{H(\R^d;k)}^2=\|v\|_{L^2(\R^d)}^2+\iint\limits_{\R^d\,\R^d}(v(x)-v(y))^2
k_s(x,y)\d y\d x\,.\]
\item $\HOmRd=\left\{ u \in \HRdk \colon u=0\text{ a.e. on
}\complement\, \Omega\right\}$
endowed with the norm $\|\cdot\|_{\HRdk}$.
\end{enumerate}  
\end{defi}
\enlargethispage{2em}
\begin{bem}\label{bem:function-spaces}{\ }
\begin{enumerate}[a)]\itemsep0.5em
 \item It is clear from this definition that
\begin{equation}\label{eq:kette-raeume}
\left(\HOmRd,\norm{\cdot}_{\HRdk} \right) \hookrightarrow \left(
\HRdk,\norm{\cdot}_{\HRdk} \right) 
\end{equation}
and $H(\R^d;k) \subset V(\Omega;k)$.
 Moreover, if $g \in \VOmk$ and $g=0$ a.e. on $\complement\,\Omega$, then
$g \in \HOmRd$.
 \item In the case $k(x,y)=\alpha (2-\alpha) \bet{x-y}^{-d-\alpha}$, $\alpha \in
(0,2)$, we write
 \[ \cE^\alpha(u,v)= \bigl[ u,v \bigr]_{\HRdk}\,. \]
 In this case we recover the standard fractional Sobolev-Slobodeckij spaces and
the norms are equivalent to the standard norms given on these spaces. More
precisely, we obtain $H^{\alpha/2}(\R^d)=\HRdk$ and if $\Omega$ is a
Lipschitz domain we have (cf. \cite[Theorem 3.33]{McLean00})
 \[ \HOmRd = H_0^{\alpha/2}(\Omega) \left( = \text{completion of }
C_c^\infty(\Omega) \text{ w.r.t. } \norm{\cdot}_{H^{\alpha/2}(\R^d)} \right).\]
If additionally $\alpha \neq 1$ then
\[ \text{completion of } C_c^\infty(\Omega) \text{ w.r.t. }
\norm{\cdot}_{H^{\alpha/2}(\R^d)} = \text{completion of } C_c^\infty(\Omega)
\text{ w.r.t. } \norm{\cdot}_{H^{\alpha/2}(\Omega)}\,. \qedhere \]
\end{enumerate}
\end{bem}
The following example illustrates that finiteness of the seminorm on $\VOmk$
requires some
regularity of the function across $\partial \Omega$:
\begin{example}\label{example:randwerte}
Let $\Omega=B_1(0)$, $\alpha\in(0,2)$ and define $k:\R^d\times\R^d\to[0,\infty]$
by 
\[k(x,y)= \alpha (2-\alpha) |x-y|^{-d-\alpha}\,.\]
In this case, $\HRdk$ coincides with the fractional Sobolev space
$H^{\alpha/2}(\R^d)$ and we have $\cE^k(u,v) = \cE^\alpha(u,v)$ for $u,v \in \HRdk$. 

Define $g:\R^d\to\R$ by
\[g(x)=\begin{cases}
        (|x|-1)^\beta\qquad &\text{if }1\leq |x|\leq 2\,,\\
        0, &\text{else.}
       \end{cases}\]
We show that $g\in \VOmk$ if and only if $\beta>\frac{\alpha-1}{2}$. Since
$\alpha$ is fixed we omit the factor $\alpha (2-\alpha)$. Then 
\begin{align*}
 [g,g]_{\VOmk}&=\iint\limits_{B_1\,\R^d} (g(x)-g(y))^2 |x-y|^{-d-\alpha}\d
x\d y=\int\limits_{B_1} \int\limits_{B_2\setminus B_1}(|x|-1)^{2\beta}
|x-y|^{-d-\alpha}\d x\d y\\
		 &=\int\limits_{B_2\setminus
B_1} (|x|-1)^{2\beta}\int\limits_{B_1}
|x-y|^{-d-\alpha}\d y\d x\,.
\end{align*}
For $1 < \bet{x} < 2$ choose $\xi = \left( \frac{3-
\bet{x}}{2\bet{x}} \right) x$. Then we may
estimate 
\begin{align*}
\int\limits_{B_1}|x-y|^{-d-\alpha}\d y&\geq \int\limits_{B_{(|x|-1)/2}(\xi)}
|x-y|^{-d-\alpha}\d y \geq \int\limits_{B_{(|x|-1)/2}(\xi)}
\left(2\left(|x|-1\right)\right)^{-d-\alpha}\d y\\
&\geq \bet{B_1} \left(\frac{|x|-1}{2}\right)^d
\left(2\left(|x|-1\right)\right)^{-d-\alpha}\,. 
\end{align*}
Thus $\displaystyle [g,g]_{\VOmk} \geq C \int\limits_{B_2\setminus B_1}
\left(|x|-1\right)^{2\beta-\alpha} \d x$ for some constant $C=C(d)>0$.
This integral is finite if $2\beta-\alpha>-1$. 
On the other hand, for $x\in B_2\setminus B_1$, we have 
\[\int\limits_{B_1}|x-y|^{-d-\alpha}\d y \leq \int\limits_{B_2(x)\setminus
B_\delta(x)} |x-y|^{-d-\alpha}\d y \leq C' \dist(x,\partial B_1)^{-\alpha}\,.\] 
Therefore, $ \displaystyle [g,g]_{\VOmk} \leq C' \int\limits_{B_2\setminus
B_1}
\left(|x|-1\right)^{2\beta-\alpha} \d x$, which shows that $ g \in \VOmk$ if and
only if $\beta>\frac{\alpha-1}{2}$.

In the second order case ($\alpha=2$) the function $g$ has a trace on
$\partial\Omega$ if and only if $\beta>\frac{1}{2}$. We note that $g\notin
H^{s}(\R^d)$ for $s>\frac12$ because of the discontinuity at $|x|=2$.  
\end{example}

%%%%%%%%%%%%
\begin{lem}\label{lem:hilbertraeume} Let $\Omega \subset \R^d$ be an open set. The spaces
$\HOmRd$ and $\HRdk$ are separable Hilbert spaces. 
\end{lem}
\begin{proof} The proof follows the argumentation in \cite[Theorem
3.1]{Wloka87}.

First we show the completeness of $\HRdk$. 
Let $(f_n)$ be a Cauchy sequence with respect to the norm
$\norm{\cdot}_{\HRdk}$. Set
\[ v_n(x,y) = (f_n(x) - f_n(y)) \sqrt{k_s(x,y)}.\]
Then, by definition of $\norm{\cdot}_{\HRdk}$ and the completeness of
$L^2(\R^d)$, $(f_n)$ converges to some $f$ in the norm of $L^2(\R^d)$. We may
chose a subsequence $f_{n_k}$ that converges a.e. to $f$. Then $v_{n_k}$
converges a.e. on $\R^d \times \R^d$ to the function
\[ v(x,y) = (f(x) - f(y)) \sqrt{k_s(x,y)}\ .\]
By Fatou's Lemma,
\[ \iint\limits_{\R^d\, \R^d} \bet{f(x) - f(y)}^2 k_s(x,y) \d x \d y \leq 
\liminf_{k \to \infty} \iint\limits_{\R^d\, \R^d} \bet{v_{n_k}(x,y)}^2 \d x \d
y \leq \sup_{k \in \N} \norm{v_{n_k}}^2_{L^2(\R^d \times \R^d)}  .\]
Since $(v_n)$ is a Cauchy sequence (and hence bounded) in $L^2(\R^d \times
\R^d)$, this shows that \linebreak[4] $\seminorm{f,f}_{\HRdk} < \infty$,
i.e. $f\in \HRdk$.
Another application of Fatou's Lemma shows that
\begin{align*}
 \seminorm{f_{n_k} - f, f_{n_k}-f}_{\HRdk} &= \iint\limits_{\R^d\, \R^d}
\bet{v_{n_k}(x,y) - v(x,y)}^2 \d x \d y\\
 &\leq \liminf_{l \to \infty} \iint\limits_{\R^d\, \R^d} \bet{v_{n_k}(x,y) -
v_{n_l}(x,y)}^2 \d x \d y \xrightarrow{k \to \infty} 0.
\end{align*}
This shows that $\norm{f_{n_k} - f}_{\HRdk} \to 0$ for $k \to \infty$ for
the subsequence $(n_k)$ chosen above and thus $\norm{f_{n}-f}_{\HRdk} \to
0$ as $n \to \infty$, since $(f_n)$ was assumed to be a Cauchy sequence. The
completeness of $\HRdk$ is proved. The completeness of $H_\Omega(\R^d;k)$ follows immediately.

The mapping $\cI$
\begin{align}
 \cI \colon \HRdk &\to L^2(\Omega) \times L^2(\Omega \times \R^d),
\label{eq:isometric}\\
 \cI(f) &= \left( f, (f(x)-f(y)) \sqrt{k_s(x,y)} \right) ,\nonumber
\end{align}
is isometric due to the definition of the norm in $\HRdk$. Having shown
the completeness of $\HRdk$ we obtain that $\cI(\HRdk)$ is a closed
subspace of the Cartesian product on the right-hand side of
\eqref{eq:isometric}. This product is separable, which implies (cf. \cite[Lemma
3.1]{Wloka87}) the separability of $\HRdk$.

Hence, $\HRdk$ is separable and so is $\HOmRd$ being a subspace
of $\HRdk$.
\end{proof}

\subsection{Variational formulation of the Dirichlet problem}

Define a bilinear form by
\begin{equation}\label{eq:bilinear-k}
 {\cE}^k(u,v)= \iint\limits_{\R^d\,\R^d} (u(x)-u(y)) v(x) k(x,y) \d y \d x\,.
\end{equation}

In order to prove well-posedness of this expression and that the bilinear form
is associated to $\opL$, we need to impose an condition on how the symmetric
part of $k$ dominates the anti-symmetric
part of $k$. We assume that there exist a symmetric
kernel $\widetilde{k}:\R^d\times\R^d\to[0,\infty]$ with $|\{y \in \R^d |\,
\widetilde{k}(x,y) = 0, k_a(x,y) \ne 0\}| = 0$ for all $x$,
and constants $A_1\geq1, A_2\geq1$ such that
\renewcommand{\theequation}{$\widetilde K$}
\begin{subequations}
\label{eq:domination_by_ktilde}
\begin{equation*}\label{eq:ktilde_comparability}
 \iint\limits_{\R^d\,\R^d}(u(x)-u(y))^2\widetilde{k}(x,y)\d x\d y\leq
 A_1 \iint\limits_{\R^d\,\R^d}(u(x)-u(y))^2k_s(x,y)\d x\d y \tag{$\widetilde{K}_1$}
\end{equation*}
for all $u\in H(\R^d;k)$, and at the same
time 
\begin{equation*}
\label{eq:ka-square-domin-by-ktilde}
   \sup_{x\in \R^d}\int\limits\frac{k_a^2(x,y)}{\widetilde{k}(x,y)}\d y\leq
A_2. \tag{$\widetilde{K}_2$}
\end{equation*} 
\end{subequations} 
\setcounter{equation}{3}
\renewcommand{\theequation}{\arabic{section}.\arabic{equation}}

A natural choice is $\widetilde{k}=k_s$ because in this case
\eqref{eq:ktilde_comparability} trivially holds and $\widetilde{k}(x,y) =
0$ implies $k_a(x,y) =0$. Assumption
\eqref{eq:domination_by_ktilde} would then reduce to the condition 
\begin{equation*}
\label{eq:ka-square-domin-by-ks} 
   \sup_{x\in \R^d}\int\limits\limits_{\{k_s(x,y)\neq
0\}} \frac{k_a^2(x,y)}{k_s(x,y)}\d y\leq A\,. \tag{K}
\end{equation*} 
Condition \eqref{eq:ka-square-domin-by-ks} appears in
\cite[(1.1)]{SchillingWang2011} and is sufficient for
that $(\cE, C_c^{0,1}(\R^d))$ extends to a regular lower bounded semi-Dirichlet
form. Note that our assumption \eqref{eq:domination_by_ktilde} is weaker and
thus we can extend \cite[Thm 1.1]{SchillingWang2011}, see
\autoref{lem:Operator-Bilinearform}. In \autoref{sec:examples} we provide an
example illustrating the difference between \eqref{eq:domination_by_ktilde} and
\eqref{eq:ka-square-domin-by-ks}.

Let us show that the bilinear form defined in \eqref{eq:bilinear-k} is
associated to $\opL$ and that the integrand in \eqref{eq:bilinear-k} is -- in
contrast to the integrand in
\eqref{eq:nonlocal-op} -- integrable in the Lebesgue sense.
\begin{lem}
\label{lem:Operator-Bilinearform}
Let $\Omega \subset \R^d$ open and assume that $k$ satisfies
\eqref{eq:integrierbarkeitsbedingung-k_s} and
\eqref{eq:domination_by_ktilde}. Define for $n \in \N$ the set $D_n = \left\{
(x,y) \in \R^d \times \R^d \colon \bet{x-y} >1/n \right\}$ and
\begin{align*}
 \cL_n u(x) &= \int_{\complement B_{1/n}(x)} \left( u(x) -u(y) \right) k(x,y) \d
y,\\
 \cE^k_n(u,v) &= \iint\limits_{D_n} \left( u(x) -u(y) \right) v(x) k(x,y) \d y \d x\,.
\end{align*}
Then we have $\left( \cL_n u,v\right)_{L^2(\R^d)} = \cE^k_n(u,v)$ and
$\lim\limits_{n \to \infty} \cE^k_n(u,v) = \cE^k(u,v)$ for all $u,v \in
C_c^\infty(\Omega)$. Moreover, $\cE^k \colon \HRdk \times \HRdk \to \R$
is continuous. By \eqref{eq:kette-raeume}, $\cE^k$ is also continuous on
$\HOmRd$.
\end{lem}
As mentioned above, our proof is an extension of the proof of \cite[Theorem
1.1]{SchillingWang2011}. 
\begin{proof} 
Assume $u,v\in C^\infty_c(\R^d)$. Splitting $k$ in its symmetric and
antisymmetric part yields
\allowdisplaybreaks
\begin{align*}
  (\cL_n u, v)_{L^2(\R^d)} &= \int\limits_{\R^d} \int\limits_{\complement
B_{1/n}(x)} (u(x)-u(y)) k(x,y)\d y\,\, v(x)\d x \\
 &= \frac{1}{2} \iint\limits_{D_n} (u(x)-u(y)) (v(x)-v(y))
k_s(x,y)\d y \d x \\
 &\qquad + \iint\limits_{D_n} (u(x)-u(y)) v(x) k_a(x,y)\d y \d
x\,. \\
\end{align*}
The first integral is finite due to
\eqref{eq:integrierbarkeitsbedingung-k_s}. 
In order to show the integrability of the second integrand we use
\eqref{eq:domination_by_ktilde} with $A=\max(A_1,A_2)$ and the
Cauchy-Schwarz inequality:
\allowdisplaybreaks
\begin{align*}
\iint\limits_{D_n} &\bet{u(x)-u(y)} \bet{v(x)} \bet{k_a(x,y)} \d y \d x\\
&= \iint\limits_{D_n} \bet{u(x)-u(y)} \bet{v(x)} \widetilde{k}^{1/2}(x,y)
\bet{k_a(x,y)} \widetilde{k}^{-1/2}(x,y)\d y \d x\\
&\leq \left(\iint\limits_{D_n} (u(x)-u(y))^2 \widetilde{k}(x,y)\d x\d
y\right)^{1/2}\left(\int\limits_{\R^d} v(x)^2 \int\limits_{\complement
B_{1/n}(x)}\frac{k_a^2(x,y)}{\widetilde{k}(x,y)}\d y \d x\right)^{1/2}\\
&\leq A  \left(\iint\limits_{D_n} (u(x)-u(y))^2 k_s(x,y)\d x\d y\right)^{1/2}
\|v\|_{L^2(\R^d)}\,.
\end{align*}
This shows $(\opL_n u, v )_{L^2(\R^d)} = {\cE}^k_n(u,v)$ and that all
expressions in this equality are well-defined. 
In particular, by dominated convergence $\lim\limits_{n \to \infty} \cE^k_n(u,v) = \cE^k(u,v)$. Moreover, $\cE^k(u,v) < \infty$ for $u,v \in \HRdk$.

Now let us prove the continuity of $\cE \colon \HRdk \times \HRdk \to
\R$. Let $u,v\in \HRdk$. Again by the symmetry of $k_s$ and by
\eqref{eq:domination_by_ktilde} we obtain
{\allowdisplaybreaks
\begin{align*}
 \bet{{\cE}^k(u,v)}&= \bet{\iint\limits_{\R^d\,\R^d} (u(x)-u(y))v(x)k(x,y) \d
y\d x}\\
 &\leq \bet{\iint\limits_{\R^d\,\R^d} (u(x)-u(y))v(x)k_s(x,y) \d y\d x}\\
 &\qquad +\iint\limits_{\R^d\,\R^d}
\bet{u(x)-u(y)}\widetilde{k}(x,y)^{1/2}\bet{v(x)}\bet{k_a(x,y)}\widetilde{k}^{
-1/2} \d y\d x\\
 &\leq \iint\limits_{\R^d\,\R^d} \bet{u(x)-u(y)}\bet{v(x)-v(y)}k_s(x,y) \d y\d
x\\
 &\qquad +\iint\limits_{\R^d\,\R^d}
\bet{u(x)-u(y)}\widetilde{k}(x,y)^{1/2}\bet{v(x)}\bet{k_a(x,y)}\widetilde{k}^{
-1/2}(x,y) \d y\d x\\
 &\leq \left(\iint\limits_{\R^d\,\R^d} (u(x)-u(y))^2k_s(x,y) \d y\d
x\right)^{1/2}\left(\iint\limits_{\R^d\,\R^d} (v(x)-v(y))^2k_s(x,y) \d y\d
x\right)^{1/2}\\
 &\qquad +A \left(\iint\limits_{\R^d\,\R^d} (u(x)-u(y))^2k_s(x,y) \d y\d
x\right)^{1/2} \|v\|_{L^2(\R^d)}\\
 &\leq C \norm{u}_{\HRdk}\norm{v}_{\HRdk}.
\end{align*}
}%end of allowdisplaybreaks
This shows that $\cE^k$ is a continuous bilinear form on $\HRdk$ and on
$\HOmRd$.
\end{proof}
%%%%%%%%%%
Finally, we are able to provide a variational formulation of the Dirichlet
problem \eqref{Dirichletproblem_einleitung} with the help of
the bilinear form $\cE^k$:
%%%%%%% 
\begin{defi}\label{defi:dirichletproblem} Assume
\eqref{eq:integrierbarkeitsbedingung-k_s} and
\eqref{eq:domination_by_ktilde}. 
Let $\Omega$ be open and bounded, $f \in \HdualOmRd$.
\begin{enumerate}[(i)]
 \item $u \in \HOmRd$ is called a solution of
\begin{equation}\tag{D$_0$} \label{eq:Dirichletproblem-nullrand}
\left\{ 
\begin{aligned}
 {\opL}u & =f \quad\text{ in } \Omega \\
 u & =0 \quad\text{ on } \complement\Omega\,,
\end{aligned} \right. 
\end{equation}
if 
\begin{equation}\label{eq:var-formulation}
   {\cE}^k(u,\varphi)=\ska{f,\varphi} \quad \text{ for all } \varphi
\in \HOmRd\,.
  \end{equation}
\item Let $g \in \VOmk$. A function $u \in \VOmk$ is called a solution of
\begin{equation}\tag{D} \label{eq:Dirichletproblem-randdaten}
\left\{ 
\begin{aligned}
 {\opL}u & =f \quad\text{ in } \Omega \\
 u & =g \quad\text{ on } \complement\Omega\,,
\end{aligned} \right. 
\end{equation}
if $u-g \in \HOmRd$ and \eqref{eq:var-formulation} holds.
\end{enumerate}
\end{defi}

In the subsequent sections we will show how to solve this problem.

\begin{bem}\label{bem:weak-solution}
If $C_c^\infty(\Omega)$ is dense in $\HOmRd$ then $\HdualOmRd$ is a
space of distributions on $\Omega$. In this case solutions in the sense of \autoref{defi:dirichletproblem} are weak solutions to \eqref{eq:Dirichletproblem-nullrand} and \eqref{eq:Dirichletproblem-randdaten}, respectively. 
\end{bem}

%%%%%
\subsection{Poincar\'{e}-Friedrichs inequality}\label{subsec:poincare}
%%%%

Let us formulate a nonlocal version of the Poincar\'{e}-Friedrichs inequality in
our set-up:

There exists a constant $C_P > 0$ such that for all $u\in\LtwoOm$
\begin{equation*}\label{eq:poincare-type-ineq}\tag{P}
\|u\|^2_{L^2(\Omega)} \leq C_P \iint\limits_{\R^d\,\R^d}(u(x)-u(y))^2k_s(x,y)\d
x\d y \,.
\end{equation*}

This inequality appears as an assumption, explicitly or implicitly, in all of
our existence results. Below, we provide
sufficient conditions on $k_s$ for \eqref{eq:poincare-type-ineq} to hold. Note
that \eqref{eq:poincare-type-ineq} may hold for integrable kernels as
well as for non-integrable kernels $k$, see the table in \autoref{sec:examples}.

%%%%%%%%%%%%%
%%%%%%%%%%%%%
%%%%%%%%%%%%%
%\section{Poincar\'{e}-Friedrichs
%inequality}\label{sec:poincare-type-inequality}
%%%%%%%%%%%%%
%%%%%%%%%%%%%
%%%%%%%%%%%%%
The following result generalizes the Poincar\'e-Friedrichs
inequalities from \cite{AksoyluMengesha2010}
and \cite[Prop.
1]{AkPa09}, respectively, to a larger class of integrable and non-integrable
kernels. (In these references, the Poincar\'e-Friedrichs inequality
is stated for functions with values in $\R^d$.) 
\begin{lem}\label{lem:poincare-nonsingular}
Let $\Omega\subset\R^d$ open and bounded and let $k:\R^d\times\R^d\to[0,\infty)$
be
measurable. Assume that there is a symmetric, a.e. nonnegative function $L \in
L^1(\R^d)$ satisfying the following properties: $\bet{ \{L > 0 \}} >0$ and there
is $c_0 > 0$ such that for all $u\in L^2(\Omega)$ 
\begin{equation}
\label{eq:vergleichbarkeit-kern-faltung}
\iint\limits_{\R^d\,\R^d} \left(u(x)-u(y)\right)^2 k(x,y)\d y \d x \geq c_0
\iint\limits_{\R^d\,\R^d} \left(u(x)-u(y)\right)^2 L(x-y) \d y \d x\,. 
\end{equation}
Then the following Poincar\'{e}-Friedrichs inequality holds: There is $C_P =
C_P(\Omega, c_0, L) >0$ such that for all $u\in \HOmRd$
 \begin{equation}\label{eq:poincare-L}
  \|u\|^2_{L^2(\R^d)}\leq C_P \iint\limits_{\R^d\,\R^d} (u(x)-u(y))^2 k(x,y)\d
y\d x\,.
 \end{equation}
\end{lem}
The example in \cite[Remark 6.20]{rossi} shows that \autoref{lem:poincare-nonsingular} fails to hold if one replaces the domain of integration $\R^d \times \R^d$ by $\Omega \times \Omega$ in \eqref{eq:vergleichbarkeit-kern-faltung} and \eqref{eq:poincare-L}.

For the proof of the Poincar\'{e}-Friedrichs inequality, we need the following
technical
Lemma taken from \cite[Lemma 10]{DydaKassmann2011}. 
\begin{lem}
 \label{lem:faltung-L1-kern}
 Let $q\in L^1(\R^d)$ be nonnegative almost everywhere and let $\supp q\subset
B_\rho(0)$ for some $\rho>0$. 
 Then for all $R>0$ and all functions $u$: 
 \begin{equation*}
    \iint\limits_{B_R\,B_R} \left(u(x)-u(y)\right)^2 \left(q*q\right)(x-y) \d
y\d x \leq 4 \|q\|_{L^1(\R^d)} \iint\limits_{B_{R+\rho}\,B_{R+\rho}}
\left(u(x)-u(y)\right)^2 q(x-y) \d y\d x\,.
 \end{equation*}
\end{lem}
%%%%%%%%%%%%%%%%%%%%
\begin{proof}[Proof of \autoref{lem:poincare-nonsingular}]
Let $L$ satisfy the assumptions of the lemma. Without loss of generality $0\in
\Omega$ \linebreak[3] (otherwise shift $\Omega$). Furthermore, we may assume
that there is $\rho>0$ such that $\supp L \subset B_\rho(0)$ (otherwise replace
$L$ by $L\,\Ind{B_\rho(0)}$). Fix $R>0$ such that $\Omega \Subset B_R(0)$. For
$\nu \in \N$ define
\[ L_\nu = \underbrace{L \ast L \ast \ldots \ast L}_{2^\nu \text{ times}} \,.\]
By the properties of $L$ we have  
\[ (L \ast L)(0)=L_1(0)=\int\limits_{\R^d} L(z)L(-z)\d z=\int\limits_{\R^d}
L^2(z)\d z >0\]
and $L_1= L \ast L\in C_b(\R^d)$, which implies that we may find $\delta>0$
(depending on $L$) such that $L_1 >0$ on $B_\delta(0)$. By the property of the
convolution there is $m \in \N$ depending on $L$ and $\Omega$ such that $L_m >
0$ on $B_R(0)$. Let $u \in \HOmRd$. Then we may estimate
\begin{align}
 E^{  L_m}_{B_R}(u,u) &:= \iint\limits_{B_R\,B_R} \left(u(x)-u(y)\right)^2  
L_m(x-y)
\d y\d x \nonumber \\
&\geq \int_{\Omega} u^2(x) \int_{\complement \Omega \cap B_R} L_m(x-y)\d y \d x
\geq C(L,\Omega) \norm{u}^2_{L^2(\R^d)}\,. \label{eq:E-Lm}
\end{align}
Iterated application of
\autoref{lem:faltung-L1-kern} (with $\rho'=2^m \rho$ and $q=L_{j}$,
$j=m-1,\ldots,0$) yields
\begin{equation}\label{eq:LemA2-iteriert}
 E^{L_m}_{B_R}(u,u)\leq 4
\|L_{m-1} \|_{L^1(\R^d)}E^{L_{m-1}}_{B_{R+ \rho'}}(u,u)\leq \ldots 
\leq 4^{m} E^{L}_{B_{R+m\rho'}}(u,u) \prod_{j=0}^{m-1}\|L_i\|_{L^1(\R^d)}\,.
\end{equation} 
\eqref{eq:E-Lm}, \eqref{eq:LemA2-iteriert} and the assumption
\eqref{eq:vergleichbarkeit-kern-faltung} imply
\begin{align*}
 \norm{u}^2_{L^2(\R^d)} &\leq \frac{1}{C(L,\Omega)} E^{L_m}_{B_R}(u,u) \leq
\frac{4^m}{C(L,\Omega)} \prod_{j=0}^{m-1}\|L_i\|_{L^1(\R^d)}
\iint\limits_{\R^d\,\R^d} (u(x)-u(y))^2 L(x-y) \d y \d x \\
 &\leq \frac{4^m}{c_0\, C(L,\Omega)} \prod_{j=0}^{m-1}\|L_i\|_{L^1(\R^d)}
\iint\limits_{\R^d\,\R^d} (u(x)-u(y))^2 k(x,y) \d y \d x\,.
\end{align*}
This finishes the proof of \autoref{lem:poincare-nonsingular}.
\end{proof}
%%%%%%%%%%%%
For non-integrable kernels $k$ we have the
following Poincar\'{e}-Friedrichs inequality:
\begin{lem}\label{lem:poincare-singulaer}
 Let $\Omega\subset\R^d$ open and bounded. Let $k:\R^d\times\R^d\to\R$ be measurable and nonnegative almost everywhere. We assume that for some
$\alpha \in (0,2)$, some $\lambda>0$ and all $u\in
L^2(\R^d)$ the kernel $k$ satisfies \eqref{eq:energie-abschaetzung-nach-unten} (see p.~\pageref*{eq:energie-abschaetzung-nach-unten}).
% Assume that there is
% $\lambda>0$ such that for all $u \in \HOmRd$
%  \begin{equation*}
%   {\cE}^{k}(u,u)\geq \lambda\, {\cE}^\alpha (u,u) \,,
%  \end{equation*}
% where 
% \begin{equation}\label{eq:standardform}
%  {\cE}^\alpha (u,v)={\cA}_{d,-\alpha}\iint\limits_{\R^d\,\R^d}
% \frac{(u(x)-u(y)) (v(x)-v(y))}{\bet{x-y}^{d+\alpha}}\d y\d x\,,\quad u,v \in
% \HOmRd
% \end{equation} 
Then there is $C_P>0$ such that for all $u\in L^2_\Omega(\R^d)$
\begin{equation}
\label{Poincare_inequality}
\|u\|_{L^2(\R^d)}\leq C_P \iint\limits_{\R^d\,\R^d}(u(x)-u(y))^2k(x,y)\d y\d
x\,.
\end{equation}
Given $\alpha_0 \in (0,2)$ and $\alpha \in [\alpha_0,2)$, the constant $C_P$
can be chosen independently of $\alpha$. 
\end{lem}
The proof of the main assertion is simple. The statement about the
independence of $C_P$ on $\alpha$ is proved in \cite[Theorem
1]{Mazya02}.

%%%%%%%%%%%
\section{G{\aa}rding inequality and Lax-Milgram Lemma}\label{sec:lax_milgram}
%%%%%%%%%%%

In this section we discuss basic properties of the bilinear form $\cE^k$ which
can be used in order to prove solvability of the Dirichlet problem. First, we
establish a G{\aa}rding inequality under the conditions
\eqref{eq:integrierbarkeitsbedingung-k_s} and \eqref{eq:domination_by_ktilde}.
Then we comment on the sector condition.
In a second subsection we show that if, in addition, the
Poincar\'{e}-Friedrichs inequality and a certain cancellation property hold, the
bilinear form $\cE^k$ is positive definite and coercive. This allows to
establish a first existence result with the help of the well-known Lax-Milgram
Lemma, \autoref{thm:existence-lax-milgram}. 

\subsection{G{\aa}rding inequality}
%%%%%%%%
\begin{lem}[G{\aa}rding inequality]\label{lem:garding}
 Let $k$ satisfy \eqref{eq:integrierbarkeitsbedingung-k_s} and
\eqref{eq:domination_by_ktilde}. Then there is $\gamma =\gamma(A_1,A_2)>0$ such
that
 \begin{equation}
    \label{eq:garding-inequality}
    {\cE}^k(u,u) \geq \frac14 \|u\|^2_{\HRdk} - \gamma \|u\|^2_{L^2(\R^d)}
\qquad \text{for all } u \in \HRdk\,.
 \end{equation} 
\end{lem}
%%%%%%%%%
\begin{proof}
Let $u\in \HRdk$ and let $A=\max\{A_1,A_2\}$. By
\eqref{eq:domination_by_ktilde} we obtain
\begin{align*}
 {\cE}^k(u,u)&\geq\frac{1}{2}\iint\limits_{\R^d\,\R^d}(u(x)-u(y))^2k_s(x,y)\d
y\d x-\iint\limits_{\R^d\,\R^d}|(u(x)-u(y))u(x)k_a(x,y)|\d y\d x\\
&\geq\frac{1}{2}\iint\limits_{\R^d\,\R^d}(u(x)-u(y))^2k_s(x,y)\d y\d x\\
&\qquad
-\iint\limits_{\R^d\,\R^d}\bet{(u(x)-u(y))\widetilde{k}^{1/2}(x,y)u(x)k_a(x,
y)\widetilde{k}^{-1/2}(x,y)}\d y\d x\\
&\geq\frac{1}{2}\iint\limits_{\R^d\,\R^d}(u(x)-u(y))^2k_s(x,y)\d y\d x \\
&\qquad - \iint\limits_{\R^d\,\R^d} \Big[
\eps\bet{u(x)-u(y)}^2\widetilde{k}(x,y)
+\frac{1}{4\eps}u^2(x) k_a^2(x,y)\widetilde{k}^{-1}(x,y) \Big] \d y\d x\\
&\geq \frac{1}{4}\iint\limits_{\R^d\,\R^d}(u(x)-u(y))^2k_s(x,y)\d y\d x
- \frac{1}{4\eps} A \|u\|^2_{L^2(\R^d)}\\
&\geq \tfrac14 \|u\|^2_{\HRdk}-\gamma \|u\|^2_{L^2(\R^d)}\,,
\end{align*}
if we choose $\eps$ sufficiently small such that $A\eps<\frac{1}{4}$
and then $\gamma=\gamma(A)$ sufficiently large.  
\end{proof}

\begin{bem}\label{bem:sector}
 As the above proof shows, assumption \eqref{eq:ka-square-domin-by-ktilde} is
tailor-made for an estimate which shows that $\cE^{k_a}$ is dominated by $\cE^{k_s}$. Thus, another consequence of
\eqref{eq:integrierbarkeitsbedingung-k_s} and \eqref{eq:domination_by_ktilde} is
the estimate
\begin{align*}
 \bet{\cE^{k_a}(u,v)} &\leq \left( \iint \left( u(x) - u(y) \right)^2
\widetilde{k}(x,y)\, \d x \d y  \right)^{1/2} \, \left( \iint v^2(x)
k_a^2(x,y)\widetilde{k}^{-1}(x,y) \, \d x \d y \right)^{1/2} \\
&\leq \sqrt{A_1} \sqrt{\cE^{k_s}(u,u)} \, \sqrt{A_2} \|v\|_{L^2} 
\end{align*}
for functions $u,v \in H(\R^d;k)$. This observation implies that the bilinear
forms $(\cE^k,H(\R^d;k))$ and  $(\cE^k,H_{\Omega}(\R^d;k))$ satisfy a sector
condition under an additional assumption, see \autoref{prop:sector-condition}
below.
\end{bem}

%%%%%%%%%%%%%%

\subsection{Application of the Lax-Milgram Lemma}
To verify that the bilinear form $\cE^k$ is positive definite, we assume
the following cancellation condition:
\begin{equation*}
 \label{eq:definitheitsbedingung} 
 \inf_{x\in\R^d} \liminf\limits_{\eps\to 0+} \int\limits_{\complement B_\eps
(x)} k_a(x,y)\d
y \geq 0. \tag{C}
\end{equation*}

\begin{bem}
As the proof below shows, assumption
\eqref{eq:definitheitsbedingung} can be relaxed. It is sufficient to assume 
 \[ \inf_{x\in\R^d} \liminf_{\eps\to0+} \int\limits_{\complement B_\eps
(x)} k_a(x,y)\d
y > -\frac{1}{2 C_P} \,,\] 
with $C_P$ as in \eqref{eq:poincare-type-ineq}. The bilinear form
$\cE^k$ would still be coercive. 
\end{bem}
%%%%%%%%
There are many cases for which condition \eqref{eq:definitheitsbedingung} holds.
If $k_a(x,y)$ depends
only on $x-y$, then for every $x \in \R^d$ and every $\eps >0$ one
obtains $\int_{\complement B_\eps
(x)} k_a(x,y)\d y =0$ which trivially implies \eqref{eq:definitheitsbedingung}.
But there are also many interesting cases for which condition
\eqref{eq:definitheitsbedingung} is not satisfied, see \autoref{sec:examples}.
%%%%%%%%
\begin{prop}\label{prop:existence-lax-milgram}
Let $\Omega\subset\R^d$ be open and bounded. Let $f\in
\HdualOmRd$ and let $k$ satisfy \eqref{eq:integrierbarkeitsbedingung-k_s}, \eqref{eq:poincare-type-ineq}, \eqref{eq:domination_by_ktilde} and \eqref{eq:definitheitsbedingung}. Then there is a unique solution $u \in \HOmRd$ to \eqref{eq:Dirichletproblem-nullrand}. 
\end{prop}
%%%%%%
\begin{proof}
In \autoref{lem:Operator-Bilinearform} it was shown that $\cE^k$ is a continuous bilinear form on $\HOmRd$. 
First, we show that \eqref{eq:definitheitsbedingung} implies that $\cE^k$ is positive definite. 
Let $u\in \HOmRd$. Observe that $k=k_s + k_a$ and for every $\eps>0$ 
\allowdisplaybreaks
\begin{align*}
 \iint\limits_{\{\bet{x-y}>\eps\}}& (u(x)-u(y))u(y)k_a(x,y)\d y\d
x=\frac{1}{2}\iint\limits_{\{\bet{x-y}>\eps\}}(u(x)-u(y))(u(x)+u(y))k_a(x,y)\d y\d x\\
 &=\frac{1}{2}\iint\limits_{\{\bet{x-y}>\eps\}}(u^2(x)-u^2(y))k_a(x,y)\d y\d x\\
 &=\frac{1}{2}\left(\,\iint\limits_{\{\bet{x-y}>\eps\}} u^2(x)k_a(x,y)\d y\d
x-\iint\limits_{\{\bet{x-y}>\eps\}}u^2(y)k_a(x,y)\d y\d x\right)\\
 &=\int\limits_{\R^d}u^2(x)\int\limits_{\complement B_\eps(x)} k_a(x,y)\d y\d x\,.
\end{align*}
From \eqref{eq:definitheitsbedingung} we obtain
\begin{align*}
\iint\limits_{\R^d\,\R^d}& (u(x)-u(y))u(y)k_a(x,y)\d y\d
x = \lim_{\eps \to 0} \iint\limits_{\{\bet{x-y}>\eps\}}& (u(x)-u(y))u(y)k_a(x,y)\d y\d
x \geq 0\,.
\end{align*}

Hence, 
\begin{equation}
\label{eq:definitheit}
  \cE^k(u,u)=\iint\limits_{\R^d\,\R^d}(u(x)-u(y))u(x)k(x,y)\d y\d x\geq \frac12
\iint\limits_{\R^d\,\R^d}(u(x)-u(y))^2k_s(x,y)\d y\d x\,,
  \end{equation}
i.e. $\cE^k(u,u)\geq 0$ for all $u\in\HOmRd$.
By \eqref{eq:poincare-type-ineq} and \eqref{eq:definitheit} 
\[ \cE^k(u,u)\geq \frac{1}{4 C_P} \|u\|^2_{L^2(\Omega)}
+ \frac{1}{4} \big[u,u\big]_{\HRdk}\geq \frac{1}{4
C_P} \norm{u}^2_{\HRdk}\,,\]
which shows that $\cE(u,u)$ is coercive. 

By the Lax-Milgram Lemma, there is a unique $u$ in $H^{k}_\Omega(\R^d)$,
such that
\begin{equation*}
  {\cE}^k(u,\varphi)=\ska{f,\varphi} \qquad \text{for all } \varphi \in
\HOmRd\,.\qedhere 
\end{equation*}
\end{proof}
Next, we show that the Dirichlet problem with suitable complement data $g$ has
also a unique solution.

\begin{thm}\label{thm:existence-lax-milgram}
Let $\Omega\subset\R^d$ be open and bounded and let $k$ satisfy
\eqref{eq:integrierbarkeitsbedingung-k_s}, \eqref{eq:poincare-type-ineq} and
\eqref{eq:definitheitsbedingung}. Assume further that there is $\widetilde k$
such that 
\begin{equation}\label{eq:ktilde_comparability-omega}
 \iint\limits_{\Omega\,\R^d}(u(x)-u(y))^2\widetilde{k}(x,y)\d x\d y\leq
 A_1 \iint\limits_{\Omega\,\R^d}(u(x)-u(y))^2 k_s(x,y)\d x\d y\,. 
 \tag{\ref*{eq:domination_by_ktilde}$_1^\prime$}
\end{equation}
for all $u\in V(\Omega;k)$ and \eqref{eq:ka-square-domin-by-ktilde} holds
true. Then \eqref{eq:Dirichletproblem-randdaten} has a unique solution $u\in
\VOmk$. 
Moreover,
\begin{equation}\label{eq:energy-estimate-solution}
  \big[u,u\big]_{\VOmk}\leq C\left(
\norm{f}^2_{\HdualOmRd}+\big[g,g\big]_{\VOmk}\right),
 \end{equation}  
 where $C=C(C_P, A_1,A_2)$ is a positive constant. 
\end{thm}
\begin{proof}
To prove the theorem we show that under the above assumptions on $g$ the problem \eqref{eq:Dirichletproblem-randdaten} 
can be transformed into a problem of the form \eqref{eq:Dirichletproblem-nullrand}. If $\widetilde{u}\in\HOmRd$ is a solution to 
\begin{equation}\label{eq:solution-reduced-boundary-data}
\left\{
\begin{array}{rll}
 \opL \widetilde u=& f - \opL g \quad &\text{in } \Omega \\
\widetilde{u}=&0 &\text{on }\complement\Omega\,
\end{array} \right.
\end{equation}
then $u=\widetilde{u}+g$ belongs to $\VOmk$ and solves \eqref{eq:Dirichletproblem-randdaten}. 
In order to apply \autoref{prop:existence-lax-milgram} to
\eqref{eq:solution-reduced-boundary-data} it remains to show that
$\opL g = {\cE}^k(g,\cdot) \in \HdualOmRd$. We have

\begin{align*}
  |{\cE}&^k(g,\varphi)|=\left|\,\iint\limits_{\R^d
\,\R^d}(g(x)-g(y))\varphi(x)k(x , y)\d y\d x\right|\\
&\leq \frac{1}{2}
\left|\,\iint\limits_{\R^d \,\R^d}(g(x)-g(y))(\varphi(x)-\varphi(y))k_s(x,y)\d
x\d y\right| +\left|\,\iint\limits_{\R^d \,\R^d}(g(x)-g(y))\varphi(x)k_a(x,y)\d
x\d y\right|\\
&=: I_1+I_2\,.
\end{align*}

Since
$\varphi =0$ a.e. on $\complement \Omega$ an application of the Cauchy-Schwarz inequality yields
\begin{align*}
  I_1&\leq \left(\iint\limits_{\Omega\,\R^d}(g(x)-g(y))^2k_s(x,y)\d y\d
x\right)^{1/2}\left( \iint\limits_{\Omega\,\R^d}
(\varphi(x)-\varphi(y))^2k_s(x,y)\d y\d x\right)^{1/2}\\
&=\big[g,g\big]^{1/2}_{\VOmk}\big[\varphi, \varphi
\big]^{1/2}_{\HRdk}\,.
\end{align*}
The term $I_2$ can be estimated as follows: By \eqref{eq:ktilde_comparability-omega} and \eqref{eq:ka-square-domin-by-ktilde}  
\begin{align*}
  I_2&\leq
\iint\limits_{\Omega\,\R^d}\left|(g(x)-g(y))\right|\widetilde{k}^{1/2}(x,
y)\left|\varphi(x)\right|\bet{k_a(x,y)}\widetilde{k}^{-1/2}(x,y)\d x\d y\\
&\leq\left(\,\iint\limits_{\Omega\,\R^d}(g(x)-g(y))^2\widetilde{k}(x,y)\d x\d
y\right)^{1/2}\left(\,\iint\limits_{\Omega\,\R^d}\varphi^2(x)\frac{k_a^2(x,y)}{
\widetilde{k}(x,y)}\d x\d y\right)^{1/2}\\
&\leq A_1^{1/2}A_2^{1/2} \big[g,g\big]^{1/2}_{\VOmk}\|\varphi\|_{L^2(\Omega)}\,.
\end{align*}
This shows the continuity of $\cE(g,\cdot) \colon \HOmRd \to \R$ and hence \eqref{eq:solution-reduced-boundary-data}
has a unique solution $\widetilde{u}\in \HOmRd$. 

In order to prove estimate \eqref{eq:energy-estimate-solution} we %set $v=u-g\in
%\begin{align*}
% {\cE}^k(v,\varphi)&=\langle f,v\rangle_{H_\Omega^*} - {\cE}^k(g,v) \\
% v&=0 \quad \text{on }\complement\,\Omega. 
%\end{align*}
apply $\widetilde{u} \in \HOmRd$ as test function and obtain
\begin{multline*}
  \langle f, \widetilde{u}\rangle_{\HdualOmRd} + {\cE}^k(g,\widetilde{u})={\cE}^k(\widetilde{u},\widetilde{u})=\iint\limits_{\R^d\,\R^d}(\widetilde{u}(x)-\widetilde{u}(y))\widetilde{u}(x)k(x,y)\d x\d y\\
 =\iint\limits_{\R^d\,\R^d}(\widetilde{u}(x)-\widetilde{u}(y))\widetilde{u}(x)k_s(x,y)\d x\d
y+\iint\limits_{\R^d\,\R^d}(\widetilde{u}(x)-\widetilde{u}(y))\,\widetilde{u}(x) k_a(x,y)\d x\d y\, ,
\end{multline*}
where the second term on the right-hand side is non-negative due to
\eqref{eq:definitheitsbedingung}. 
Hence,
\begin{multline*}
 \frac{1}{2}\iint\limits_{\R^d\,\R^d}(\widetilde{u}(x)-\widetilde{u}(y))^2k_s(x,y)\d x\d y \leq \ska{
f,\widetilde{u}}_{\HdualOmRd} + \cE^k(g,\widetilde{u}) \\
 \leq \norm{f}_{\HdualOmRd} \norm{\widetilde{u}}_{\HRdk} +\!
\iint\limits_{\R^d\,\R^d} (g(x)-g(y))\widetilde{u}(x) k_s(x,y)\d x\d y \\+
\iint\limits_{\R^d\,\R^d} (g(x)-g(y))\widetilde{u}(x) k_a(x,y)\d x\d y\,.
\end{multline*}
The Young inequality and the fact that $v = 0$ a.e. on $\complement \Omega$
imply for $\eps>0$, to be specified later,
\begin{multline*}
 \iint\limits_{\R^d\,\R^d} (g(x)-g(y))\widetilde{u}(x) k_s(x,y)\d x\d y \leq
\iint\limits_{\Omega\,\R^d} \bet{g(x)-g(y)} \bet{\widetilde{u}(x)-\widetilde{u}(y)} k_s(x,y) \d x \d y
\\
\leq \frac{1}{4\epsilon} \iint\limits_{\Omega\,\R^d} \bet{g(x)-g(y)}^2 k_s(x,y)
\d y \d x + \eps \iint\limits_{\R^d\,\R^d} \bet{\widetilde{u}(x)-\widetilde{u}(y)}^2 k_s(x,y) \d x \d
y\,.
\end{multline*}
Similarly, using \eqref{eq:ktilde_comparability-omega} and
\eqref{eq:ka-square-domin-by-ktilde} we obtain
\begin{align*}
 \iint\limits_{\R^d\,\R^d}& (g(x)-g(y))\widetilde{u}(x) k_a(x,y)\d x\d y \leq
\iint\limits_{\Omega\,\R^d} \bet{g(x)-g(y)} \bet{\widetilde{u}(x)} \bet{k_a(x,y)} \d y \d x
\\
&\leq \frac{A}{4\epsilon} \iint\limits_{\Omega\,\R^d} \bet{g(x)-g(y)}^2
\widetilde{k}(x,y) \d y \d x + \frac{\eps}{A} \iint\limits_{\R^d\,\R^d}
\widetilde{u}^2(x) \frac{k^2_a(x,y)}{\widetilde{k}(x,y)}  \d y \d x \\
&\leq \frac{A^2}{4\epsilon} \iint\limits_{\Omega\,\R^d} \bet{g(x)-g(y)}^2
k_s(x,y) \d y \d x + \eps \|\widetilde{u}\|^2_{L^2(\Omega)} 
\,.
\end{align*}
Altogether we obtain
\begin{align*}
\iint\limits_{\R^d\,\R^d} &(\widetilde{u}(x)-\widetilde{u}(y))^2k_s(x,y)\d x\d y 
 \leq 2 \norm{f}_{\HdualOmRd} \norm{\widetilde{u}}_{\HRdk} \\
&\quad + \frac{1}{2\epsilon} \iint\limits_{\Omega\,\R^d} \bet{g(x)-g(y)}^2
k_s(x,y)
\d y \d x + 2 \eps \iint\limits_{\R^d\,\R^d} \bet{\widetilde{u}(x)-\widetilde{u}(y)}^2 k_s(x,y) \d x
\d y \\
&\quad+ \frac{A^2}{2\epsilon} \iint\limits_{\Omega\,\R^d} \bet{g(x)-g(y)}^2
k_s(x,y) \d y \d x + 2 \eps \|\widetilde{u}\|^2_{L^2(\Omega)}\\ 
& \leq \frac{1}{2\eps} \norm{f}^2_{\HdualOmRd} 
+ \left(\frac{A^2 +1}{2\epsilon} \right)  \big[g,g\big]_{\VOmk} + 4 \eps
\|\widetilde{u}\|^2_{\HRdk} \,.
\end{align*}

Applying the Poincar\'e-Friedrichs inequality \eqref{eq:poincare-type-ineq}, and
choosing
$\eps = \frac{1}{16 C_P}$ we deduce
\begin{equation*}
\iint\limits_{\R^d\,\R^d}(\widetilde{u}(x)-\widetilde{u}(y))^2k_s(x,y)\d x\d y\leq
c_1 \norm{f}^2_{\HdualOmRd}  + c_2 \big[g,g\big]_{\VOmk}\,,
\end{equation*}
where $c_1\geq 1$ depends on $C_P$ and $c_2\geq 1$ depends on $A$ and $C_P$.
Since $\HOmRd\subset\VOmk$ and $u=\widetilde{u}+g$ the assertion \eqref{eq:energy-estimate-solution} follows. \qedhere
\end{proof}

As mentioned in \autoref{bem:sector}, condition \eqref{eq:definitheitsbedingung} ensures the
sector condition in the sense of \cite{MaRo92}:

\begin{prop}\label{prop:sector-condition}
 Let $k$ satisfy \eqref{eq:integrierbarkeitsbedingung-k_s}, 
\eqref{eq:domination_by_ktilde} and \eqref{eq:definitheitsbedingung}. Then the
bilinear forms
$(\cE^k,H(\R^d;k))$ and  $(\cE^k,H_{\Omega}(\R^d;k))$ satisfy the weak
sector condition. If, in addition, $k$ satisfies
\eqref{eq:poincare-type-ineq},
then the bilinear form $(\cE^k,H_{\Omega}(\R^d;k))$ satisfies the strong sector
condition. 
\end{prop}

\section{Weak maximum principle and Fredholm
alternative}\label{sec:uniqueness-fredholm}
The goal in this section is to prove existence and uniqueness when the bilinear form $\cE^k$ is no longer positive definite. To this end we establish a weak maximum principle implying that the homogeneous equation has only the trivial solution. Then we apply Fredholm's alternative.
\subsection{Weak maximum principle}
In this subsection we prove a weak maximum principle when the kernels $k$
exhibit a non-integrable singularity at the diagonal. We need to impose two
different assumptions on the class of admissible kernels.

We assume that for some $\alpha \in (0,2)$, some $\lambda>0$ and all $u\in
L^2(\R^d)$ the estimate  
\begin{equation}\label{eq:energie-abschaetzung-nach-unten}\tag{E$_\alpha$}
 \iint\limits_{\R^d\,\R^d} (u(x)-u(y))^2k_s(x,y) \d y\d x \geq \lambda \
\alpha (2-\alpha) \iint\limits_{\R^d\,\R^d}
\frac{(u(x)-u(y))^2}{|x-y|^{d+\alpha}} \d y\d x
\end{equation}
holds true. Further, we assume that there is $D>1$ such that for almost every
$x,y\in\R^d$ \begin{equation}
\label{eq:sym-echt-groesser-k-a}
\bet{k_a(x,y)} \leq D^{-1} k_s(x,y) 
\end{equation}

Condition \eqref{eq:energie-abschaetzung-nach-unten} requires some minimal
singularity of $k_s$ at the diagonal. It is not restrictive when considering
non-integrable kernels. Concerning
condition \eqref{eq:sym-echt-groesser-k-a} note that, by definition, the
inequality $\bet{k_a(x,y)} \leq k_s(x,y)$ holds  for almost every
$x,y\in\R^d$. Condition \eqref{eq:sym-echt-groesser-k-a} is satisfied by
several examples, e.g for 
\[k(x,y)=|x-y|^{-d-\alpha}+g(x,y) \mathbbm{1}_{B_1}(x-y)|x-y|^{-\d-\beta} \,,\]
if $0<\beta <\alpha/2$ and $\|g\|_\infty \leq \frac12$, cp. Example
\eqref{exa:singular-nonsymmetric} in \autoref{sec:examples}. But there are
also examples which violate the condition, e.g. $k(x,y) =
|x-y|^{-d-\alpha}\mathds{1}_{\R^d_+}(x-y)$, cp.
Example \eqref{exa:standard-singular-nonsymmetric}. 

Under the above conditions we can prove the
following weak maximum principle: 
\begin{thm} \label{thm:weak-maximum}
Let $k$ satisfy  \eqref{eq:integrierbarkeitsbedingung-k_s},
\eqref{eq:ka-square-domin-by-ks}, \eqref{eq:energie-abschaetzung-nach-unten},
\eqref{eq:sym-echt-groesser-k-a}. 
Let $u\in \HOmRd$ satisfy 
\begin{align}\label{eq:def-u-subs}
{\cE}^k(u,\varphi)\leq 0   \quad \text{ for all } \varphi\in
 \HOmRd \,. 
\end{align}
Then $\sup_{\Omega} u\leq 0$.
\end{thm}
%%%%%%%%%%
\begin{bem}
As the proof reveals, it is possible to weaken assumption
\eqref{eq:sym-echt-groesser-k-a} significantly because the estimate under
consideration is needed only in an integrated sense. However, it seems
challenging to provide a simple appropriate alternative to
\eqref{eq:sym-echt-groesser-k-a}.
\end{bem}
%%%%%
\begin{proof}
We apply a strategy which is often used in the proof of the
weak maximum principle for second order differential operators (e.g. \cite{GilbargTrudinger77}). We
first show that $u$ attains its supremum on a set of
positive measure. In a second step we show that this leads to a contradiction
if the supremum is positive. 

We choose as
test function $v=(u-k)^+$, where $0\leq k< \sup_{\Omega} u$. Then
$v\in \HOmRd$ and 
 \begin{equation}\label{eq:supsolution}
  {\cE}^k(u,v)=\iint\limits_{\R^d\,\R^d} (u(x)-u(y))v(x) k(x,y) \d y\d x \leq
0
 \end{equation}
Since $(u(x)-u(y))v(x) = \left[ (u-k)_+(x) + (u-k)_+(y) + (u-k)_-(y) \right] v(x)$ we deduce
\begin{align*}
 \frac{1}{2}\iint\limits_{\R^d\,\R^d}& (v(x)-v(y))^2 k_s(x,y) \d y\d x+
\iint\limits_{\R^d\,\R^d} (u(y)-k)^-(u(x)-k)^+ k(x,y) \d y\d x\\
 &\leq -\iint\limits_{\R^d\,\R^d} (v(x)-v(y))v(x) k_a(x,y) \d y\d x\,,
\end{align*}
and since the second term on the left hand side is positive  
\[\frac{1}{2}\iint\limits_{\R^d\,\R^d} (v(x)-v(y))^2 k_s(x,y) \d y\d x\leq
-\iint\limits_{\R^d\,\R^d} (v(x)-v(y))v(x) k_a(x,y) \d y\d x\, .\]
Now by Cauchy-Schwarz and the assumptions on $k_a$  
\begin{align*}
 &\iint\limits_{\R^d\,\R^d} (v(x)-v(y))v(x) k_s(x,y) \d y\d x \\
 &\qquad \leq 2 \iint\limits_{\R^d\,\R^d} \bet{v(x)-v(y)}k_s^{1/2}(x,y)
\bet{v(x)}k_s^{-1/2}(x,y)\bet{k_a(x,y)}\d y\d x\\
 &\qquad \leq  2A^{1/2} \|v\|_{L^2(\R^d)} \left(\iint\limits_{\R^d\,\R^d}
(v(x)-v(y))^2 k_s(x,y) \d y\d x\right)^{1/2}\,,
\end{align*}
or equivalently
\begin{equation*}
 \left(\iint\limits_{\R^d\,\R^d} (v(x)-v(y))^2 k_s(x,y) \d y\d
x\right)^{1/2}\leq 2 {A^{1/2}} \|v\|_{L^2(\R^d)}.
\end{equation*}
By \eqref{eq:energie-abschaetzung-nach-unten} and since $v=0$ on
$\complement\Omega$, the
Sobolev and the Hölder inequality imply that there is a constant $C=C(d)>0$
such that 
\[ \|v\|_{L^{2d/(d-\alpha)}(\R^d)}\leq C \|v\|_{L^2(\Omega)} \leq  C |\supp
v|^{\alpha/2d}\|v\|_{L^{2d/(d-\alpha)}(\Omega)}.\]
(If $d \leq 2$ the critical exponent may be replaced by any number greater than $2$.)
Thus 
\[|\supp v| \geq C^{-2d/\alpha}.\]
This inequality is independent of $k$ and therefore it holds for
$k\nearrow\sup_\Omega u$. Therefore $u$ must attain its supremum on a set of
positive measure. This completes the first step of the proof.

We now derive a contradiction. Without loss of generality we may assume 
$\sup_{\Omega}u=1$. Set $v=u^+$. 
We define a new function $\overline{v}$ by 
\[\overline{v}=\frac{v}{1-v}=\frac{1}{1-v}-1.\]
We want to use $\overline{v}$ as a test function in \eqref{eq:def-u-subs} but
it is not clear whether $\overline{v}$ belongs to $\HOmRd$. Thus we look at
approximations and define for small $\epsilon > 0$
\[v_\eps = (1-\eps) v \quad \text{ and } \quad \overline v_\eps =
\frac{v_\eps}{1-v_\eps} \,. \] 
The function $\overline{v}_\epsilon$ is an admissible test function. However, in
order to simplify the presentation, we use $\overline{v}$ instead
of $\overline{v}_\epsilon$ and postpone this issue until the end of the
proof. Plugging $\overline{v}$ into \eqref{eq:supsolution}, we obtain 
\begin{align*}
 \frac{1}{2}\iint\limits_{\R^d\,\R^d}&
(v(x)-v(y))\left(\frac{1}{1-v(x)}-\frac{1}{1-v(y)}\right) k_s(x,y) \d y\d x\\
 &\leq -\iint\limits_{\R^d\,\R^d} (v(x)-v(y))\frac{v(x)}{1-v(x)} k_a(x,y) \d y\d
x\\
 &=-\frac{1}{2}\iint\limits_{\R^d\,\R^d} (v(x)-v(y))\left(\frac{v(x)}{1-v(x)}+
\frac{v(y)}{1-v(y)}\right)k_a(x,y) \d y\d x\,.
\end{align*}
This is equivalent to 
\begin{align*}
 &\iint\limits_{\R^d\,\R^d} \frac{(v(x)-v(y))^2}{(1-v(x))(1-v(y))} k_s(x,y) \d
y\d x\\
 &\leq - \iint\limits_{\R^d\,\R^d}
\frac{(v(x)-v(y))}{(1-v(x))^{1/2}(1-v(y))^{1/2}}\left(\frac{
v(x)(1-v(y))+v(y)(1-v(x))}{(1-v(x))^{1/2}(1-v(y))^{1/2}}\right)k_a(x,y) \d y\d
x\,.
\end{align*}
An application of the Young inequality leads to  
\begin{multline*}
 \iint\limits_{\R^d\,\R^d} \frac{(v(x)-v(y))^2}{(1-v(x))(1-v(y))} k_s(x,y) \d
y\d x \leq \frac{1}{2} \iint\limits_{\R^d\,\R^d}
\frac{(v(x)-v(y))^2}{(1-v(x))(1-v(y))} k_s(x,y)\d y\d x\\
 +\frac{1}{2}\iint\limits_{\R^d\,\R^d}
\frac{[v(x)(1-v(y))+v(y)(1-v(x))]^2}{(1-v(x))(1-v(y))}\frac{k_a^2(x,y)}{
k_s(x,y)} \d y\d x
\end{multline*}
and hence 
\begin{align}\label{eq:maxprin_start}
\begin{split}
 \frac{1}{2}\iint\limits_{\R^d\,\R^d} &\frac{(v(x)-v(y))^2}{(1-v(x))(1-v(y))}
k_s(x,y) \d y\d x\\ 
 &\leq \frac{1}{2}\iint\limits_{\R^d\,\R^d}
\frac{[v(x)(1-v(y))+v(y)(1-v(x)))^2]}{(1-v(x))(1-v(y))}\frac{k_a^2(x,y)}{
k_s(x,y)} \d y\d x\,.
\end{split}
\end{align}
Using $v=0$ on $\complement\Omega$, $v\leq 1$ and that
$\frac{k_a^2(x,y)}{k_s(x,y)}$ is symmetric, the right hand side can be
estimated from above as follows:
\begin{align*}
\frac{1}{2}&\iint\limits_{\R^d\,\R^d}
\frac{[v(x)(1-v(y))+v(y)(1-v(x)))^2]}{(1-v(x))(1-v(y))}\frac{k_a^2(x,y)}{
k_s(x,y)} \d y\d x \\
&\leq \iint\limits_{\Omega\,\R^d} \left( \frac{v^2(x)(1-v(y))}{1-v(x)} +
v(x)v(y) \right)\frac{k_a^2(x,y)}{k_s(x,y)} \d y\d x\\
&\leq \theta^2 \iint\limits_{\Omega \,\R^d}
\frac{(v(x)-v(y))^2}{(1-v(x))(1-v(y))}\frac{k_a^2(x,y)}{k_s(x,y)} \d y\d x
+ \left(\tfrac{\theta}{\theta-1} +1\right) \iint\limits_{\Omega\,\R^d}
\frac{k_a^2(x,y)}{k_s(x,y)} \d y\d x \\
&\leq \frac{\theta^2}{D} \iint\limits_{\R^d \,\R^d}
\frac{(v(x)-v(y))^2}{(1-v(x))(1-v(y))} k_s(x,y) \d y\d x
+ \left(\tfrac{\theta}{\theta-1} +1\right) A |\Omega|\,,
\end{align*}
where we have applied \autoref{lem:abschaetzung-RS-weak-maximum} and
\eqref{eq:sym-echt-groesser-k-a}. Now, we choose $\theta=\sqrt{\frac{D+1}{2}}$
such that $\frac{\theta^2}{D} <1$. Combining the above estimate and
\eqref{eq:maxprin_start} leads to 
\[\iint\limits_{\R^d\,\R^d} \frac{(v(x)-v(y))^2}{(1-v(x))(1-v(y))} k_s(x,y) \d
y\d x\leq c_1 A |\Omega|\,,\]
for some positive constant $c_1=c_1(D)$. Next, we want to estimate the
left-hand side from below. We apply the inequality
\[\frac{(a-b)}{ab}=(a-b)(b^{-1}-a^{-1})\geq (\log a-\log b)^2\,,\] 
which holds for positive reals $a,b$, to $a=1-v(y)$ and $b=1-v(x)$. Thus we
obtain 
\[\iint\limits_{\R^d\,\R^d} \Big(\log (1-v(x))-\log (1-v(y))\Big)^2
k_s(x,y) \d
y\d x \leq c_1 A |\Omega| \,.\]
Due to condition \eqref{eq:energie-abschaetzung-nach-unten} we can apply the
Sobolev inequality and obtain  
\[ \|w\|_{L^{2d/(d-\alpha)}}\leq c_2 A |\Omega| \,, \] 
where $c_2 \geq 1$ and $w= \log(1-v)$. Recall that, in fact, we have proved
$\|w_\epsilon\|_{L^{2d/(d-\alpha)}}\leq c_2 A |\Omega|$ for $w_\epsilon=
\log(1-v_\epsilon)$ and every $\eps \in (0,\tfrac12)$, where $c_2$ is
independent of
$\epsilon$. 

By Fatou's lemma, this contradicts the fact that $v=u^+$ attains
is supremum $1$ on a set of positive measure. The proof is complete. 
\end{proof}
\begin{lem}\label{lem:abschaetzung-RS-weak-maximum}
Assume $\theta>1$ and $a, b \in [0,1)$. Then 
 \begin{align}
  \label{eq:abschaetzung-RS-weak-maximum}
  \frac{1-a}{1-b}&\leq \theta^2 \frac{\left(b-a\right)^2}{(1-b)(1-a)}+
\frac{\theta}{\theta-1} \\
  \label{eq:abschaetzung-RS-weak-maximum-sym}
\frac{1-b}{1-a} +  \frac{1-a}{1-b}&\leq 2 \theta^2
\frac{\left(b-a\right)^2}{(1-b)(1-a)}+
\frac{2 \theta}{\theta-1} 
\end{align}
\end{lem}
\begin{proof} It is sufficient to establish assertion
\eqref{eq:abschaetzung-RS-weak-maximum} since it implies
\eqref{eq:abschaetzung-RS-weak-maximum-sym}. For the proof of 
\eqref{eq:abschaetzung-RS-weak-maximum} it is sufficient to assume $a \leq b$. 
Assume $\theta>1$ and $0\leq a\leq b<1$. Then, for $t=\frac{b}{a}$  
 \[0\leq a <1 \leq t <\frac{1}{a} \]
and inequality \eqref{eq:abschaetzung-RS-weak-maximum} reads 
\begin{align}\label{eq:clever_aim}
\frac{1-a}{1-ta}\leq \theta^2 \frac{a^2(t-1)^2}{(1-ta)(1-a)}+
\frac{\theta}{\theta-1}\,. 
\end{align}

\textit{Case 1:} $\frac{1}{a}-1\leq \theta (t-1)$. In this case  
\begin{align*}
\left(\tfrac{1}{a}-1\right)^2\leq \theta^2 (t-1)^2 \quad \Rightarrow \quad
(1-a)\leq \theta^2 \, \frac{a^2(t-1)^2}{(1-a)} \quad \Rightarrow
\quad\frac{1-a}{1-ta}\leq \theta^2  \frac{a^2(t-1)^2}{(1-ta)(1-a)}\,,
\end{align*}
which proves \eqref{eq:clever_aim}.

\textit{Case 2:} $\frac{1}{a}-1>  \theta (t-1) \; \Leftrightarrow \; (t-1)<
\frac{1}{\theta}\left(\frac{1}{a}-1\right)  \; \Leftrightarrow
\; t<\frac{1}{\theta}\left(\frac{1}{a}-1\right)+1$. Therefore  
\begin{align*}
 \frac{1-a}{1-ta}=\frac{a\left(\frac{1}{a}-1\right)}{a\left(\frac{1}{a}-t\right)
}
 \leq
\frac{\frac{1}{a}-1}{\frac{1}{a}-1-\frac{1}{\theta}\left(\frac{1}{a}-1\right)} =
\frac{\frac{1}{a}-1}{(\frac{1}{a}-1)(1-\frac{1}{\theta})}=\frac{\theta}{\theta-1
} \,,
\end{align*}
which again proves \eqref{eq:clever_aim}.
\end{proof}

%%%%
\subsection{Fredholm alternative}\label{subsec:fredholm}
%%%%
The aim of this subsection is to prove existence and uniqueness of solutions to
\eqref{eq:Dirichletproblem-randdaten} without assuming positive definiteness of
the bilinear form $\cE^k$, i.e. without assuming 
\eqref{eq:definitheitsbedingung}. As for the weak maximum principle we assume
the kernel $k$ to satisfy \eqref{eq:energie-abschaetzung-nach-unten}
and \eqref{eq:sym-echt-groesser-k-a}. Note that \eqref{eq:energie-abschaetzung-nach-unten} implies \eqref{eq:poincare-type-ineq} by \autoref{lem:poincare-singulaer}.
%%%%%%

\pagebreak[2]
We prove the following well-posedness result:
%%%%%%
\begin{thm}\label{thm:Fredholm}
 Let $\Omega\subset\R^d$ be open and bounded. Let $f\in
\HdualOmRd$ and let $k$ satisfy \eqref{eq:integrierbarkeitsbedingung-k_s}, 
 \eqref{eq:ka-square-domin-by-ks}, \eqref{eq:sym-echt-groesser-k-a} and
\eqref{eq:energie-abschaetzung-nach-unten}. 
Then the Dirichlet problem \eqref{eq:Dirichletproblem-randdaten} has a unique solution
$u \in V(\Omega;k)$. Moreover, there is a constant $C=C(C_P,A,D)>0$ such that
\begin{equation}
 \label{eq:energy-estimate-solution-fredholm}
  \big[u,u\big]_{V(\Omega;k)}\leq C\left(
\|f\|^2_{\HdualOmRd}+\|g\|^2_{L^2(\Omega)}+\big[g,g\big]_{V(\Omega;k)}+\|u\|^2_{L^2(\Omega)}
\right)\,.
 \end{equation}  
\end{thm}
%%%%%%%%%% Liste für die Steps %%%%%%%%%%%%
\newlist{steps}{enumerate}{1}
\newlength{\tempStep}
\settowidth{\tempStep}{\textbf{Step 3:}}
\setlist[steps]{label=Step \arabic{*}:, font=\bfseries, labelwidth=\tempStep, leftmargin=!, itemindent=3em}
%%%%%%%%%%%%%%%%%%%%%%%%%%%%%%%%%%%%%%%%%%%
\begin{proof}
We use the Fredholm alternative (see e.g. \cite{evans}). 
\begin{steps}
 \item We will use \eqref{eq:energie-abschaetzung-nach-unten} to show that the
embedding $\HOmRd\hookrightarrow L^2_\Omega(\R^d)$ is compact. Since the embedding $L^2(\Omega) \hookrightarrow \HdualOmRd$ is continuous we obtain then the compactness of the embedding $\HOmRd\hookrightarrow \HdualOmRd$.\\
Let $\cA\subset \HOmRd$ with $\norm{u}_{\HRdk}\leq C$ for all $u\in \cA$ and
some $C<\infty$.
Let $B\subset \R^d$ be an open ball with $\Omega \subset B$. Let us recall that
the embedding
$H^{\alpha/2}(B)\hookrightarrow L^2(B)$ is compact. Then, for
$u\in \cA$
\begin{align*}
 \norm{u}^2_{H^{\alpha/2}(B)}&=\iint\limits_{B\,B}
(u(x)-u(y))^2\bet{x-y}^{-d-\alpha}\d y\d x \\
&\leq  \iint\limits_{R^d\,R^d} (u(x)-u(y))^2\bet{x-y}^{-d-\alpha}\d y\d x \\
&\leq \lambda^{-1} \iint\limits_{R^d\,R^d} (u(x)-u(y))^2 k_s(x,y) \d y\d x \leq \lambda^{-1} C^2 \,,
\end{align*}
where we used \eqref{eq:energie-abschaetzung-nach-unten}.
Therefore $\cA$ is bounded in $H^{\alpha/2}(B)$ and thus pre-compact in
$L^2(B)$.
By the definition of $\HOmRd$ we know $u=0$ on $\complement\Omega$ and thus the
set $\cA$ is also pre-compact in $L^2_\Omega(\R^d)$ and in $\HdualOmRd$.
%%%
\item Existence and uniqueness of \eqref{eq:Dirichletproblem-nullrand}. By \autoref{lem:garding} the bilinear form \[ (u,v) \mapsto \cE^k(u,v) + \gamma (u,v)_{L^2(\Omega)}\] is coercive for some $\gamma = \gamma(A)>0$ and therefore there is a unique solution $u \in \HOmRd$ to the problem  
\begin{equation} \label{eq:hilfsproblem}
 \left\{ 
  \begin{aligned}
   \cE^k(u,v) + \gamma (u,v)_{L^2(\Omega)} &= \ska{f,v} \quad &\text{for all } v \in \HOmRd,\\
   u &= 0\qquad &\text{on } \complement \Omega\,.
  \end{aligned} \right.
\end{equation}
Moreover, due to \autoref{lem:garding} the solution $u$ satisfies
\[ \norm{u}^2_{\HOmRd} \leq 4 \cE^k(u,u) + 4 \gamma \norm{u}_{L^2(\Omega)}^2 = 4 \ska{f,v} \leq 4 \norm{f}_{\HdualOmRd} \norm{u}_{\HOmRd}\,.
\]
This estimate together with Step 1 shows that the operator $K \colon \HdualOmRd \to \HdualOmRd$, which maps the inhomogeneity $f$ to the solution $u \in \HOmRd \Subset \HdualOmRd$ of \eqref{eq:hilfsproblem}, is a compact operator. Fredholm's theorem in combination with the weak maximum principle \autoref{thm:weak-maximum} shows that \eqref{eq:Dirichletproblem-nullrand} has a unique solution $u \in \HOmRd$.
\item The well-posedness of \eqref{eq:Dirichletproblem-randdaten} follows in the same way as in the proof of \autoref{thm:existence-lax-milgram}. It remains to prove the estimate \eqref{eq:energy-estimate-solution-fredholm}. Let $u$ be the solution of \eqref{eq:Dirichletproblem-randdaten}. We apply
$v=u-g \in \HOmRd$ as test function:
\begin{align*}
  &\langle f, v\rangle_{\HdualOmRd} - {\cE}^k(g,v)={\cE}^k(v,v)\\
  &=\iint\limits_{\R^d\,\R^d}(v(x)-v(y))v(x)k(x,y)\d x\d y\\
 &=\iint\limits_{\R^d\,\R^d}(v(x)-v(y))v(x)k_s(x,y)\d x\d
y+\iint\limits_{\R^d\,\R^d}(v(x)-v(y))\,v(x) k_a(x,y)\d x\d y\, ,
\end{align*}
As in the proof of \autoref{thm:existence-lax-milgram} we may
estimate 
\begin{align*}
&\iint\limits_{\R^d\,\R^d} (v(x)-v(y))^2k_s(x,y)\d x\d y 
 \leq \frac{1}{2\epsilon} \norm{f}^2_{\HdualOmRd} + 2\epsilon \norm{v}^2_{\HRdk} \\
&\quad +
\frac{1}{2\epsilon} \iint\limits_{\Omega\,\R^d} \bet{g(x)-g(y)}^2 k_s(x,y)
\d y \d x + 2 \eps \iint\limits_{\R^d\,\R^d} \bet{v(x)-v(y)}^2 k_s(x,y) \d x
\d
y \\
&\quad + \frac{A^2}{2\epsilon} \iint\limits_{\Omega\,\R^d} \bet{g(x)-g(y)}^2
k_s(x,y) \d y \d x + 2 \eps \|v\|^2_{L^2(\Omega)}\\
&\quad+ 2\iint\limits_{\R^d\,\R^d}\bet{v(x)-v(y)}\bet{v(x)} \bet{k_a(x,y)}\d x\d
y\,. 
\end{align*}
Due to $v=0$ on $\complement \Omega$, the last term can be estimated as follows: 
\begin{align*}
 &\iint\limits_{\R^d\,\R^d}\bet{v(x)-v(y)}\bet{v(x)} \bet{k_a(x,y)}\d x\d y\\
 &\leq \eps \iint\limits_{\R^d\,\R^d} (v(x)-v(y))^2 k_s(x,y) \d x \d y +
\frac{A}{4\eps} \|v\|^2_{L^2(\Omega)}\,. 
\end{align*}
Hence, after choosing $\epsilon$ appropriately,
\begin{equation*}
\iint\limits_{\Omega\,\R^d}(v(x)-v(y))^2k_s(x,y)\d x\d y\leq
 \norm{f}^2_{\HdualOmRd}  + c_1 \big[g,g\big]_{V(\Omega;k)}+ c_2
\|v\|^2_{L^2(\Omega)} \,,
\end{equation*}
where $c_1,c_2 >0$ depend on $A$. This implies
\eqref{eq:energy-estimate-solution-fredholm}. \qedhere
\end{steps}
\end{proof}
%%%%%%%%%%%
%%%%%%%%%%%
%%%%%%%%%%%
\section{Parabolic Problem}\label{sec:parabolic}
In this section we prove well-posedness of the initial boundary value problem
\begin{equation}\label{eq:dirichlet-para}
 \left\{
\begin{array}{rll}
 \partial_t u + \opL u =& f \qquad  &\text{in } (0,T) \times \Omega ,\\
 u =& 0  &\text{on } [0,T] \times \complement \Omega,\\
 u =& u_0   &\text{on } \{0\} \times \Omega,
\end{array} 
\right.
\end{equation}
where $\Omega \subset \R^d$ is a bounded domain, $Q_T= (0,T) \times \Omega$ and
\begin{equation}
 f \in L^2(Q_T), \qquad u_0 \in L^2(\Omega).
\end{equation}
To put this problem in a functional analytic framework we define for a Banach space $\cB$
\begin{align*}
 W(0,T) = W(0,T;\cB) &= \left\{ u \in L^2\left(0,T;\cB\right) \colon u' \text{ exists and } u'
\in L^2\left(0,T;\cB^*\right) \right\}\,.
\end{align*}
$W(0,T)$ is a Banach space endowed with the norm
\[ 
 \norm{u}^2_{W(0,T)} = \int_0^T \norm{u(t)}^2_{\cB} \d t + \int_0^T
\norm{u'(t)}^2_{\cB^*} \d t\,. 
\]
We consider the Gelfand triplet
\[ \HOmRd \hookrightarrow L^2(\Omega) \hookrightarrow \HdualOmRd\,,\]
where the two embeddings are continuous and each space is dense in the following
one.

Define for $u,v \in \cB$
\begin{equation}
 \cE^k(t;u,v) = \iint\limits_{\R^d\,\R^d} \left( u(x)-u(y) \right) v(x) k_t(x,y)
\d x \d y\,,
\end{equation}
where we assume that 
\begin{equation}\label{eq:zeitabh-kern}
 k_t(x,y)=a(t,x,y) k(x,y)
\end{equation}
with a measurable function $a \colon \R \times \R^d \times \R^d \to
\left[\frac12 ,1 \right]$ which is symmetric with respect to $x$ and $y$, and a
measurable kernel $k \colon \R^d \to \R^d \to [0,\infty]$.

\begin{defi}[Parabolic variational formulation]\label{prob:abstract} Let $u_0 \in L^2(\Omega)$ and\\ $f \in
L^2(0,T;\HdualOmRd)$. We say that $u \in W(0,T;\HOmRd)$ is a solution of
\eqref{eq:dirichlet-para} if for all $v \in \HOmRd$
\begin{subequations}\label{eq:abstract-initial}
 \begin{align}
\frac{\d}{\d t} (u(t),v)_{L^2(\Omega)} + \cE^k(t;u(t),v) &= \ska{f,v}\quad
\text{ for a.e. } t \in (0,T)\,, \label{eq:abstract-initial-a}\\
 u(0)&=u_0\,. \label{eq:abstract-initial-b}
\end{align}
\end{subequations}
\end{defi}
We refer to \autoref{defi:dir-par-nonzero} for the corresponding problem with nonzero complement data.
\begin{bem}{\ }
\begin{enumerate}[a)]
 \item \eqref{eq:abstract-initial-a} is satisfied if and only if for all $v \in
\HOmRd$ and for all $\phi \in C_c^\infty(0,T)$
\[
  -\int_0^T (u(t),v)_{L^2(\Omega)} \phi'(t) \d t + \int_0^T \cE^k(t;u(t),v)
\phi(t) \d t 
  = \int_0^T \ska{f,v} \phi(t) \d t\,.\]
\item The initial condition \eqref{eq:abstract-initial-b} is well-defined due to
the embedding \[ W(0,T) \hookrightarrow C([0,T];L^2(\Omega))\,,\] see e.g. \cite{Wloka87}. 
\item \autoref{bem:weak-solution} holds analogously.
\qedhere
\end{enumerate}
\end{bem}
A well-known parabolic analog (e.g. \cite[Corollary 23.26]{Zeidler90}) of the
Lax-Milgram Lemma asserts that there is a unique solution to
\eqref{eq:dirichlet-para} provided the bilinear form $\cE(t;\cdot,\cdot) : \cB
\times \cB \to \R$ has the following properties:
\begin{subequations}
\begin{align}
&\parbox{0.8\textwidth}{For all $u,v \in \cB$ the mapping $t \mapsto \cE^k(t;u,v)$
is measurable on $(0,T)$,} \label{eq:measurable}\\
&\parbox{0.8\textwidth}{there is $M> 0$ such that for all $t \in (0,T)$ and $u,v
\in \cB$ we have
$$ \bet{\cE^k(t;u,v)} \leq M \norm{u}_\cB \norm{v}_\cB,$$} \label{eq:bounded} \\
&\parbox{0.8\textwidth}{there are $m> 0$ and $L_0 \geq 0$ such that for all $t
\in (0,T)$ and $u \in \cB$ we have
$$ \cE^k(t;u,u) \geq m \norm{u}_\cB^2 - L_0 \norm{u}^2_{L^2(\Omega)}.$$}
\label{eq:garding}
\end{align}
\end{subequations}
The measurability condition \eqref{eq:measurable} follows immediately from the
measurability of $k_t$. Due to the special structure \eqref{eq:zeitabh-kern} of
$k_t$ we may refer to \autoref{lem:Operator-Bilinearform} for the proof of the
continuity condition \eqref{eq:bounded} and to \autoref{lem:garding} for the
G{\aa}rding inequality \eqref{eq:garding}. Therefore we obtain the following
result:
%%%%%%%%
\begin{thm}[Well-posedness of the parabolic
problem]\label{thm:wellposedness-para}
Assume that $k_t$ is of the form \eqref{eq:zeitabh-kern}, where $k$ satisfies
\eqref{eq:integrierbarkeitsbedingung-k_s} and \eqref{eq:domination_by_ktilde}.
Then there is a unique solution $u \in W(0,T;\HOmRd)$ of \eqref{eq:dirichlet-para}. Moreover, for all $u_0 \in L^2(\Omega)$ and $f \in
L^2(0,T;\HdualOmRd)$ there is a constant $C>0$ such that
\begin{equation}
 \norm{u}_{W(0,T)} \leq C \left( \norm{f}_{L^2(0,T;\HdualOmRd)} +
\norm{u_0}_{L^2(\Omega)} \right)\,.
\end{equation}
The constant $C$ depends on the constants $M,m$ and $L_0$ in
\eqref{eq:bounded},\eqref{eq:garding}.
\end{thm}
%%%%%%%%%%
In order to consider non-homogeneous complement data we assume
that these values are prescribed by a function $g \colon [0,T] \times \R^d \to \R$. For the functional analytic treatment of this problem we need a modification of the space $V(\Omega;k)$ that turns this linear space into a Hilbert space. One possibility is to define
\begin{equation}
 H(\Omega;k) = \left\{ g \in L^2(\R^d) \colon \left( g(x)-g(y) \right) k_s^{1/2}(x,y) \in L^2(\Omega \times \R^d) \right\}\,. 
\end{equation}
From the proof of \autoref{lem:hilbertraeume} it can be seen that this space is a separable Hilbert space with the inner product defined by
\[ \left( g,h \right)_{H(\Omega;k)} = (g,h)_{L^2(\R^d)} + \left[ g,h\right]_{V(\Omega;k)}\,.\]
\begin{defi}[Parabolic variational formulation]\label{defi:dir-par-nonzero} Let $f \in L^2(Q_T)$ and $g \in W(0,T;\HOmk)$. We say that $u \in W(0,T;\HOmk)$ is a solution of
\begin{equation}\label{eq:dirichlet-para-nonzero}
 \left\{
\begin{array}{rll}
 \partial_t u + \opL u =& f \qquad  &\text{in } (0,T) \times \Omega ,\\
 u =& g  &\text{on } [0,T] \times \complement \Omega ,\\
 u =& g(0,\cdot)   &\text{on } \{0\} \times \Omega\,,
\end{array} 
\right.
\end{equation}
if \eqref{eq:abstract-initial-a} holds, if $u-g \in W(0,T;\HOmRd)$ and if $u(0)=g(0) \in \HOmk$.
\end{defi}
\begin{cor}\label{cor:wellposedness-para-nonzero}
Assume that $k_t$ is of the form \eqref{eq:zeitabh-kern}, where $k$ satisfies
\eqref{eq:integrierbarkeitsbedingung-k_s} and \eqref{eq:domination_by_ktilde}.
Then there is a unique solution $u \in W(0,T;\HOmk)$ of \eqref{eq:dirichlet-para-nonzero}.
\end{cor}
%%%%%%%
\begin{proof}
 By \autoref{thm:wellposedness-para} there is a unique solution $u \in W(0,T;H_\Omega(\R^d;k))$ to the problem
\begin{equation}
 \left\{
\begin{array}{rll}
 \partial_t u + \opL u =& f - \partial_t g - \opL g \qquad  &\text{in } (0,T) \times \Omega ,\\
 u =& 0  &\text{on } [0,T] \times \complement \Omega ,\\
 u =& 0   &\text{on } \{0\} \times \Omega.
\end{array} 
\right.
\end{equation}
Then $\widetilde u = u + g$ satisfies $\widetilde u \in W(0,T;H(\Omega;k))$ and \eqref{eq:dirichlet-para-nonzero}.
\end{proof}
Let us emphasize that the conditions on the complement data $g$ are far from being optimal. On the one hand, it would be desirable to relax the condition on the spatial decay imposed by $g(t,\cdot) \in L^2(\complement \Omega)$ as in \autoref{sec:lax_milgram}. On the other hand, one could try to follow the program as in the case of second order parabolic equations (e.g. \cite{LiMa72}) in order to relax the regularity assumptions of $g$ with respect to $t$ and $x$.
%%%
%%%
\section{Examples of kernels}\label{sec:examples}
%%%
%%%
We provide several examples of kernels $k:\R^d\times\R^d\to[0,\infty]$ to which
the theory above can be applied. Recall that all kernels studied in this work
satisfy assumption \eqref{eq:integrierbarkeitsbedingung-k_s}. Further, we
distinguish two cases. We call a kernel $k$ {\bf integrable} if, for every
$x\in \R^d$ the quantity $\int_{\R^d} k_s(x,y) \d y$ is finite and the  
mapping $x\mapsto \int_{\R^d} k_s(x,y) \d y$ is locally integrable. We
call a kernel {\bf non-integrable} if it is not integrable in the sense above.
At the end of this section we list all examples together with their
corresponding properties.
 
\subsection{Integrable kernels}

Let us start with a simple observation. Every kernel with the property that the
antisymmetric part is of the form $k_a(x,y)=g(x-y)$ for some function $g$
satisfies the assumption \eqref{eq:definitheitsbedingung}. This follows from
the fact that for $x \in \R^d$ 
\begin{equation*}
 \int\limits_{\complement B_\eps (x)}k_a(x,y)\d
y=\int\limits_{\complement B_\eps (x)}g(x-y) \d
y=\int\limits_{\complement B_\eps (0)}g(z)dz=0 \,.
\end{equation*}

\begin{enumerate}
 \item \label{exa:ball-1} $k(x,y) := \mathbbm{1}_{B_1}(x-y)$. The kernel is
obviously symmetric. Thus it satisfies \eqref{eq:definitheitsbedingung}. It
also satisfies the Poincar\'{e}-Friedrichs inequality
\eqref{eq:poincare-type-ineq} as
shown in \autoref{subsec:poincare}.

\item \label{exa:B2-ohne-B1} $k(x,y) := \mathds{1}_{B_R\setminus
B_r}(x-y)$ for some numbers $0 < r < R$. Again,
\eqref{eq:definitheitsbedingung} and the Poincar\'{e}-Friedrichs inequality
\eqref{eq:poincare-type-ineq} hold. 

\item \label{exa:ball-anti} $k(x,y) :=
\mathds{1}_{B_1 \cap \R^d_+}(x-y)$. Symmetrization leads to
\begin{align*}
k_s(x,y) &= \tfrac12 \mathds{1}_{B_1 \cap \R^d_+}(x-y) +  \tfrac12
\mathds{1}_{B_1 \cap \R^d_+}(y-x) = \tfrac12  \mathds{1}_{B_1}(x-y) \\
k_a(x,y) &= \tfrac12 \mathds{1}_{B_1 \cap \R^d_+}(x-y) -  \tfrac12
\mathds{1}_{B_1 \cap
\R^d_-}(x-y) \,. 
\end{align*}
Since $k$ depends only on $x-y$, condition \eqref{eq:definitheitsbedingung}
holds. Concerning the Poincar\'{e}-Friedrichs inequality
\eqref{eq:poincare-type-ineq},
$k$ is not different from example \eqref{exa:ball-1}.

\item \label{exa:cones-integrable} This example is more general than
Example \ref{exa:ball-anti}. Set $k(x,y) := \mathds{1}_{B_1}(x-y)
\mathds{1}_{\mathcal{C}}(x-y)$ where the set $\mathcal{C}$ is defined by
$\mathcal{C}=\{h \in \R^d
|\, \tfrac{h}{|h|} \in I\}$ and $I$ is an arbitrary
nonempty open subset of $S^{d-1}$. If $I$ is of the form $I=B_r(\xi) \cap
S^{d-1}$ for
some $\xi \in S^{d-1}$ and some $r >0$, then $\mathcal{C}$ is a cone. In any
case, we obtain 
\begin{align*}
k_s(x,y) &= \tfrac12 \mathds{1}_{B_1 \cap (\mathcal{C} \cup
-\mathcal{C})}(x-y)\,, \\
k_a(x,y) &= 
\tfrac12 \mathds{1}_{B_1 \cap \mathcal{C}}(x-y) - \tfrac12
\mathds{1}_{B_1 \cap -\mathcal{C}}(x-y) \,. 
\end{align*}

\end{enumerate}

In the examples above, $k(x,y)$ depends only on $x-y$. As a
result, one can choose $L(z)=k(0,y-x)$ in the condition
\eqref{eq:vergleichbarkeit-kern-faltung}. Let us look at examples where this
is not possible.

\begin{enumerate}
\setcounter{enumi}{4}
 
\item \label{exa:integrable-mit-funktion-g} $k(x,y):=g(x,y)
\mathds{1}_{B_1}(x-y)$, where $g$ is any measurable bounded function satisfying
$g \geq c$ almost everywhere for some constant $c >0$. Note
that $g$ does not need to be symmetric. Then 
\begin{align*}
k_s(x,y) &= \tfrac12 (g(x,y)+g(y,x)) \mathds{1}_{B_1}(x-y) \\
k_a(x,y) &= \tfrac12 (g(x,y)-g(y,x)) \mathds{1}_{B_1}(x-y)  \,. 
\end{align*}
Condition \eqref{eq:definitheitsbedingung} does not hold in general but
\eqref{eq:ka-square-domin-by-ks} holds because $k_s(x,y) \geq
c \mathds{1}_{B_1}(x-y)$ which allows us to apply the
Poincar\'{e}-Friedrichs inequality
\eqref{eq:poincare-type-ineq} choosing $L(z) = c \mathds{1}_{B_1}(z)$ in
\eqref{eq:vergleichbarkeit-kern-faltung}.  
\end{enumerate}

\begin{enumerate}
\setcounter{enumi}{5}
\item \label{exa:definitheit-gilt-nicht} Here, we set $d=1$ and define a
kernel $k:\R\times\R\to[0,\infty)$ as follows. Define 
\[ D= [-1,0] \times [0,1] \cup \{(x,y) \in \R^2 |\, (x \leq y \leq x+1)\}  \,. \]
Set $k(x,y):= 2 \cdot
\mathds{1}_{D}(x,y)$. Then the antisymmetric part of $k$ is given by 
\[k_a(x,y)= \mathds{1}_{D}(x,y) - \mathds{1}_{(-D)}(x,y)\,. \]
Due to the construction of $D$ we obtain for $|x|>1$ $\lim\limits_{\eps \to 0+}
\int_{\complement B_\eps (x)}k_a(x,y)\d y = 0$ whereas for $x \in (-1,1)$ we
obtain $\lim\limits_{\eps \to 0+}
\int_{\complement B_\eps (x)}k_a(x,y)\d y = -x$, which implies that $k$
does
not satisfy condition \eqref{eq:definitheitsbedingung}. Though, conditions
\eqref{eq:ka-square-domin-by-ks} and \eqref{eq:vergleichbarkeit-kern-faltung}
hold true because of $k_s(x,y) \geq \mathds{1}_{B_1}(x-y)$. 

\item \label{exa:definitheit-ohne-x-y} Again, set $d=1$. We define $k(x,y)$ by
$k(x,y) = 2\cdot\mathds{1}_{(-4,4)}(x-y) + k_a(x,y)$ where
\[k_a(x,y)=\begin{cases}
          g(x-1,y-3) &\text{if }x<y\\
          -g(y-1,x-3) &\text{else ,}
         \end{cases} \]
and
\[g(x,y)=\sgn(xy)\mathbbm{1}_{(-1,1)\times(-1,1)}(x,y) \,.\]
By construction $k_a$ is antisymmetric and satisfies condition
\eqref{eq:definitheitsbedingung}. 
Thus $k$ is not a function of  $x-y$ but still
satisfies \eqref{eq:definitheitsbedingung}. Conditions
\eqref{eq:ka-square-domin-by-ks} and \eqref{eq:vergleichbarkeit-kern-faltung}
hold true, too.
\end{enumerate}

\subsection{Non-integrable kernels} 
Here are several examples of kernels $k:\R^d \times \R^d \to
[0,\infty]$ with a singularity at the diagonal. See above for our definition of 
when we call a kernel non-integrable. Recall that we want all
examples to satisfy \eqref{eq:integrierbarkeitsbedingung-k_s}. Throughout this
section (with one exception) $\alpha \in (0,2)$ is an arbitrary fixed number. 
\begin{enumerate}[resume]
 \item  \label{exa:standard-singular} $k(x,y):= |x-y|^{-d-\alpha}$. 
 Obviously, $k$ is symmetric and satisfies
\eqref{eq:integrierbarkeitsbedingung-k_s}. Conditions 
\eqref{eq:definitheitsbedingung} and \eqref{eq:ka-square-domin-by-ks} hold due to the
symmetry. \autoref{lem:poincare-singulaer} can be directly applied. This
kernel $k$ is very special because the space $\HRdk$ is isomorphic to the
fractional Sobolev space
$H^{\alpha/2}(\R^d)$ (cf. \autoref{bem:function-spaces}b). There is a constant $C\geq 1$, independent of $\alpha$,
such that for all $v \in C^\infty_c(\R^d)$
\[ C^{-1} \|v\|_{H^{\alpha/2}(\R^d)} \leq
\alpha(2-\alpha) \|v\|_{\HRdk} \leq C
\|v\|_{H^{\alpha/2}(\R^d)} \,,\]
where $\|v\|^2_{H^{\alpha/2}(\R^d)} = \int (1+|\xi|^2)^{\alpha/2} |\widehat{v}(\xi)|^2
\d
\xi$. Thus, for fixed $v \in
C^\infty_c(\R^d)$,
\begin{align*}
\alpha(2-\alpha)
\left[v,v\right]_{\HRdk}&\longrightarrow\left[v,v\right]_{H^1(\R^d)} & \text{for
}&\alpha\rightarrow 2^-\\
 \alpha(2-\alpha) {\norm v}_{\HRdk}&\longrightarrow{\norm
v}_{L^2(\R^d)}&\text{for
}&\alpha\rightarrow 0^+.
\end{align*}
Similar results hold true for $\R^d$ replaced by a bounded
domain \cite{Brezis_Mironescu02,Mazya02}.

%%%
\item \label{exa:standard-singular-symm-cones} Let $I$ be an arbitrary
nonempty open subset of $S^{d-1}$ with the property $I=-I$. Set $\mathcal{C}=\{h
\in \R^d
|\, \tfrac{h}{|h|} \in I\}$ and $k(x,y):= |x-y|^{-d-\alpha}
\mathds{1}_\mathcal{C}(x-y)$. Again, $k$ is symmetric and satisfies
\eqref{eq:integrierbarkeitsbedingung-k_s}. It turns out that $k$ is comparable
to example \eqref{exa:standard-singular} in the sense of
\autoref{lem:poincare-singulaer}. The only difference is that the constant
$\lambda$ depends on $I$.

%%%
\item \label{exa:standard-singular-nonsymmetric}
$k(x,y) := |x-y|^{-d-\alpha}\mathds{1}_{\R^d_+}(x-y)$. This example is
different from example \eqref{exa:standard-singular-symm-cones} because $k$ is
not symmetric anymore. The symmetric and antisymmetric parts are given by
\begin{align*}
k_s(x,y) &= \tfrac12 |x-y|^{-d-\alpha} \\
k_a(x,y) &= |x-y|^{-d-\alpha} \big( \tfrac12 \mathds{1}_{\R^d_+}(x-y) - 
\tfrac12 \mathds{1}_{\R^d_-}(x-y)\big) \,. 
\end{align*}
\autoref{lem:poincare-singulaer} can still be applied but conditions
\eqref{eq:ka-square-domin-by-ks} and \eqref{eq:domination_by_ktilde} do
not hold. Condition \eqref{eq:definitheitsbedingung} does hold, though.

%%%
 \item \label{exa:singular-nonsymmetric} Assume $0 < \beta <
\frac{\alpha}{2}$ and $g:\R^d \times \R^d \to [-K,L]$ measurable for some $K,L
>0$. Define 
 \[k(x,y):=|x-y|^{-d-\alpha}+g(x,y) \mathbbm{1}_{B_1}(x-y)|x-y|^{-d-\beta}\]
Additionally, we assume that $k$ is nonnegative. This property does not
follow in general under the assumptions above. However, for every choice of $K$
there are many admissible cases with $\inf g = -K$. We obtain
$k_s(x,y)\geq
\frac{1}{2} |x-y|^{-d-\alpha}$ for $|x-y| \leq (2K)^\frac{-1}{\alpha-\beta}$.
Since $k_s$ is nonnegative, we can apply \autoref{lem:poincare-singulaer} and 
\eqref{eq:poincare-type-ineq} holds. Further 
 \eqref{eq:domination_by_ktilde} is satisfied with
$\widetilde{k}(x,y)=|x-y|^{-d-\alpha} \mathbbm{1}_{B_1}(x-y)$ and
$A_1=1$. Conditions \eqref{eq:definitheitsbedingung} and
\eqref{eq:ka-square-domin-by-ks} hold for some but not for all choices of $g$. 
\end{enumerate}

The following example is an extension and, at the same time, a special case of
Example \eqref{exa:singular-nonsymmetric}. Example \eqref{exa:singular-cones}
shows that our
condition \eqref{eq:domination_by_ktilde} is indeed a relaxation of
\eqref{eq:ka-square-domin-by-ks} or \cite[(1.1)]{SchillingWang2011}.

\begin{enumerate}[resume]
%%%
\item \label{exa:singular-cones} Assume $0 < \beta <
\frac{\alpha}{2}$. Let $I_1, I_2$ be arbitrary nonempty disjoint
open subsets of $S^{d-1}$ with $I_1=-I_1$ and
$|-I_2 \setminus I_2 | > 0$.  Set
$\mathcal{C}_j=\{h \in \R^d
|\, \tfrac{h}{|h|} \in I_j\}$ for $j \in \{1,2\}$. Set 
\[k(x,y)=
|x-y|^{-d-\alpha}\mathds{1}_{\cC_1}(x-y)
+|x-y|^{-d-\beta}\mathds{1}_{\cC_2}(x-y) \mathds{1}_{B_1}(x-y)\,.\]
The symmetric and antisymmetric parts of $k$ are given by 
\begin{align*}
k_s(x,y) &=
|x-y|^{-d-\alpha}\mathds{1}_{\cC_1}(x-y)
+\tfrac12 |x-y|^{-d-\beta}\mathds{1}_{\cC_2 \cup (-\cC_2)}(x-y)
\mathds{1}_{B_1}(x-y)\,, \\ 
k_a(x,y) &= \tfrac12 |x-y|^{-d-\beta}\mathds{1}_{\cC_2\cap B_1}(x-y) - \tfrac12
|x-y|^{-d-\beta}\mathds{1}_{(-\cC_2)\cap B_1}(x-y)\,.
\end{align*}
Let us show that condition \eqref{eq:ka-square-domin-by-ks} does not hold, i.e.
\eqref{eq:ka-square-domin-by-ktilde} is not satisfied
for $\widetilde{k}=k_s$. Let $h, h_a, h_s, \widetilde{h}:\R^d \to [0,\infty]$ be
defined
by $h(x-y)=k(x,y)$ and $h_a, h_s, \widetilde{h}$ accordingly. Note that $\bet{h_a} = h_s$ on $\cC_2 \cup -\cC_2$. Then 
\begin{align*}
\sup_{x\in \R^d}&\int\limits\limits_{\{k_s(x,y)\neq
0\}} \frac{k_a(x,y)^2}{k_s(x,y)}\d y =  \int\limits_{\R^d}
\frac{h_a^2(z)}{h_s(z)} \mathds{1}_{B_1}(z) dz = \int\limits_{\{\cC_2\cup
(-\cC_2)\}}
\frac{h_a^2(z)}{h_s(z)} \mathds{1}_{B_1}(z) dz \\
&=\int\limits_{\{\cC_2\cup (-\cC_2)\}}h_s(z) \mathds{1}_{B_1}(z) dz
=\int\limits_{\{\cC_2\cup
(-\cC_2)\}} \frac12  |z|^{-d-\beta} \mathds{1}_{B_1}(z) dz = +\infty 
\end{align*}
Let us explain why
\eqref{eq:ktilde_comparability} and \eqref{eq:ka-square-domin-by-ktilde} hold
for $\widetilde{k}(x,y)=|x-y|^{-d-\alpha}$. \eqref{eq:ktilde_comparability}
follows easily from $k_s(x,y) \geq |x-y|^{-d-\alpha}\mathds{1}_{\cC_1}(x-y)$,
the constant $A_1$ needs to be chosen in dependence of $\cC_1$ resp. $I_1$. 
Let us check \eqref{eq:ka-square-domin-by-ktilde}:
\begin{align*}
\sup_{x\in \R^d} \int \frac{k_a^2(x,y)}{\widetilde{k}(x,y)}\d y
&=  \int\limits_{\R^d} \frac{h_a^2(z)}{\widetilde{h}(z)}dz
 =\int\limits_{\{\cC_2\cup
-(\cC_2)\}}\frac{h_a^2(z)}{\widetilde{h}(z)} \mathds{1}_{B_1}(z) dz \\
&= \int\limits_{\{\cC_2\cup
(-\cC_2)\}} \frac14  |z|^{-d-2\beta+\alpha} \mathds{1}_{B_1}(z) dz \leq A_2 \,, 
\end{align*}
where $A_2$ depends on $I_2$ and $\alpha/2 - \beta$. {\bf Note:} If we modify
the example by choosing $I_1 = S^{d-1}$, i.e.
$\cC_1=\R^d$, then condition \eqref{eq:ka-square-domin-by-ks} does hold. 

\item \label{exa:cusp} The following example appears in \cite[Example
12]{DydaKassmann2011}. Assume $0 < b < 1$ and $0 < \alpha' < 1+
\frac1b$. Define $\Gamma=\{(x_1,x_2)\in\R^2| \bet{x_1}\geq \bet{x_2}^b \text{ or
}
\bet{x_2}\geq \bet{x_1}^b\}$ and set 
\[k(x,y)= \mathds{1}_{\Gamma\cap B_1}(x-y)\bet{x-y}^{-d-\alpha'}\,. \]
Note that the kernel $k$ depends only on $x-y$ and is symmetric but condition 
\eqref{eq:integrierbarkeitsbedingung-k_s} is not obvious. Using integration in
polar coordinates one can show that there is $C\geq 1$ such that for $\alpha =
\alpha' - (1/b -1)$, $h(z)=k(x,x+z)$ and every $r \in (0,1)$
\[ r^{2} \int\limits_{B_r} |z|^2 h(z) \d z + \int\limits_{\R^d
\setminus B_r} h(z) \d z \leq C r^{-\alpha} \,. \]
Thus $\alpha$ is the effective order of differentiability of
the corresponding integro-differential operator. In \cite{DydaKassmann2011}
the comparability of the quadratic forms needed for
\autoref{lem:poincare-singulaer} is established. From the point of view of this
article the kernel $k$ is very similar to the kernel $|x-y|^{-d-\alpha}
\mathds{1}_{B_1}(x-y)$. Of course, one can now produce related nonsymmetric
examples. 
\item  \label{exa:schilling} The following example is taken from \cite{FuUe12},
\cite{SchillingWang2011}. It provides a nonsymmetric kernel $k$ with a
singularity on the diagonal which is non-constant. Assume
$0<\alpha_1\leq\alpha_2<2$ and let $\alpha:\R^d \to [\alpha_1, \alpha_2]$ be a
measurable function. We assume that $\alpha$ is continuous and that the modulus
of continuity $\omega$ of the function $\alpha$ satisfies 
 \[ \int\limits_0^1\frac{(\omega(r)|\log r|)^2}{r^{1+\alpha_2}}dr<\infty \,.\]
Note that, as a result, there are $\beta \in (0,1)$ and $C_H>0$ such that $[
\alpha ]_{C^{0,\beta}(\R^d)} \leq C_H$. Let $b:\R^d \to \R$ be another
measurable function which is bounded between two positive constants and
satisfies $|b(x)-b(y)| \leq c |\alpha(x) - \alpha(y)|$ as long as $|x-y|\leq 1$
for some constant $c>0$. Finally, set
 \[k(x,y)=b(x)|x-y|^{-d-\alpha(x)} \,.\]
In \cite{SchillingWang2011} it is proved, that $k$ satisfies
\eqref{eq:integrierbarkeitsbedingung-k_s} and \eqref{eq:ka-square-domin-by-ks}.
Since $\alpha$ is bounded from below by $\alpha_1$, the
Poincar\'{e}-Friedrichs inequality
\eqref{eq:poincare-type-ineq} holds. Condition \eqref{eq:sym-echt-groesser-k-a}
does not hold for this example since
$\displaystyle{\lim_{\bet{x-y}\to\infty}\frac{\bet{k_a(x,y)}}{k_s(x,y)}=1}$.\\
Let us slightly modify the example and look at $k'(x,y) = \mathbbm{1}_{B_R}(x-y)
k(x,y)$ for some $R \gg 1$. Then
conditions
\eqref{eq:integrierbarkeitsbedingung-k_s}, \eqref{eq:ka-square-domin-by-ks}
and \eqref{eq:poincare-type-ineq} still hold true for $k'$. 

\begin{lem}The kernel $k'$ satisfies \eqref{eq:sym-echt-groesser-k-a}.
\end{lem}

\begin{proof}  We have to show
that $\frac{\bet{k'_a(x,y)}}{k'_s(x,y)}\leq \Theta<1$
for all $x,y\in\{\bet{x-y}<R\}$. By assumption there are
$c_1,c_2>0$ such that
\[ c_1\leq b(x)\leq c_2 \quad \text{ for all } x\in\R^d\,. \] 
We can assume that
$\alpha(x)\leq \alpha(y)$
due to the symmetry of $\frac{\bet{k'_a(x,y)}}{k'_s(x,y)}$.\\
%%%%%%%
\textit{Case 1a:} $\bet{x-y}\leq1$ and $k'_a(x,y)>0$.    
Then
\begin{align*}
 \frac{\bet{k'_a(x,y)}}{k'_s(x,y)}&=
\frac{b(x)|x-y|^{-d-\alpha(x)}-b(y)|x-y|^{-d-\alpha(y)}}{b(x)|x-y|^{-d-\alpha(x)
}+b(y)|x-y|^{-d-\alpha(y)}} \\
 &=\frac{b(x)-b(y)|x-y|^{\alpha(x)-\alpha(y)}}{b(x)+b(y)|x-y|^{
\alpha(x)-\alpha(y)}} \leq
\frac{1-\frac{c_1}{c_2}}{1+\frac{c_1}{c_2}}=:\Theta_1
\end{align*}
\textit{Case 1b:} $\bet{x-y}\leq1$ and $k'_a(x,y)<0$.  
Then
\begin{align*}
\phantom{xxx} \frac{\bet{k'_a(x,y)}}{k'_s(x,y)}&=
\frac{b(y)|x-y|^{-d-\alpha(y)}-b(x)|x-y|^{-d-\alpha(x)}}{b(x)|x-y|^{-d-\alpha(x)}+b(y)|x-y|^{-d-\alpha(y)}}
 =\frac{b(y)-b(x)|x-y|^{\alpha(y)-\alpha(x)}}{b(y)+b(x)|x-y|^{\alpha(y)-\alpha(x)}}\\                                  
\end{align*}
Since $\bet{\alpha(y)-\alpha(x)}\leq C_H \bet{x-y}^\beta$, we obtain
\[ |x-y|^{\alpha(y)-\alpha(x)}\geq \bet{x-y}^{C_H \bet{x-y}^\beta}\geq \delta(C_H, \beta)>0.\] Thus
\begin{equation*}
 \frac{\bet{k'_a(x,y)}}{k'_s(x,y)} \leq
\frac{1-\frac{c_1}{c_2}\delta}{1+\frac{c_1}{c_2}\delta}=: \Theta_2
\end{equation*}
\textit{Case 2a:} \,$1<\bet{x-y}<R$ and $k'_a(x,y)<0$.  
Then
\begin{align*}
 \frac{\bet{k'_a(x,y)}}{k'_s(x,y)}&=
\frac{b(y)|x-y|^{-d-\alpha(y)}-b(x)|x-y|^{-d-\alpha(x)}}{b(x)|x-y|^{-d-\alpha(x)
}+b(y)|x-y|^{-d-\alpha(y)}} \\
 &=\frac{b(y)-b(x)|x-y|^{\alpha(y)-\alpha(x)}}{b(y)+b(x)|x-y|^{
\alpha(y)-\alpha(x)}} \leq \frac{1-\frac{c_1}{c_2}}{1+\frac{c_1}{c_2}}=\Theta_1
\end{align*}
\textit{Case 2b:} \,$1<\bet{x-y}<R$ and $k'_a(x,y)>0$.    
Then
\begin{align*}
 \frac{\bet{k'_a(x,y)}}{k'_s(x,y)}&=
\frac{b(x)|x-y|^{-d-\alpha(x)}-b(y)|x-y|^{-d-\alpha(y)}}{b(x)|x-y|^{-d-\alpha(x)
}+b(y)|x-y|^{-d-\alpha(y)}} \\
 &=\frac{b(x)-b(y)|x-y|^{\alpha(x)-\alpha(y)}}{b(x)+b(y)|x-y|^{
\alpha(x)-\alpha(y)}} \leq  \frac{1-\frac{c_1}{c_2}
R^{\alpha_1-\alpha_2}}{1+\frac{c_1}{c_2} R^{\alpha_1-\alpha_2} }=:\Theta_3 \qedhere
\end{align*}
\end{proof}
We have shown than $k'$ satisfies all conditions needed in order to apply
\autoref{thm:Fredholm}. 
\end{enumerate}

Although we have provided
several different examples, our class is still rather small. All examples of
non-integrable kernels from above relate, in one way or another,
to the standard kernel
$|x-y|^{-d-\alpha}$ for some $\alpha \in (0,2)$ and the Sobolev-Slobodeckij
space $H^{\alpha/2}(\R^d)$. We could also study examples with kernels which
relate to a generic standard kernel $|x-y|^{-d} \phi(|x-y|^2)^{-1}$ where
$\phi$ itself can be chosen from a rather general class of
functions, e.g. the class of complete Bernstein functions.

Let us summarize the examples from above in a table. Recall that all examples
satisfy \eqref{eq:integrierbarkeitsbedingung-k_s}. In the tabular below, the
symbol {\bf ?} indicates that the answer depends on
the concrete specification of the example.

\smallskip

\begin{center}
   \begin{tabular}{ c  c  c  c  c  c  c  c  c  c  c  c 
c  c  c  c }
\hline
{\bf Examples:} & \eqref{exa:ball-1} &
\eqref{exa:ball-anti} & \eqref{exa:B2-ohne-B1} &
\eqref{exa:cones-integrable} & \eqref{exa:integrable-mit-funktion-g} &
\eqref{exa:definitheit-gilt-nicht} & \eqref{exa:definitheit-ohne-x-y} & \vdots &
\eqref{exa:standard-singular} & \eqref{exa:standard-singular-symm-cones} &
\eqref{exa:standard-singular-nonsymmetric} &  \eqref{exa:singular-nonsymmetric}
& \eqref{exa:singular-cones} &    \eqref{exa:cusp} & \eqref{exa:schilling} \\

\toprule

\eqref{eq:poincare-type-ineq} & \checkmark & \checkmark & \checkmark &
\checkmark & \checkmark & \checkmark & \checkmark  & \vdots &
\checkmark &  \checkmark & \checkmark & \checkmark & \checkmark & \checkmark &
\checkmark \\
\hline
\eqref{eq:definitheitsbedingung} & \checkmark & \checkmark & \checkmark &
\checkmark & {\bf ?} & \njet & \checkmark  & \vdots &
\checkmark &  \checkmark & \checkmark & {\bf ?}  & \checkmark &
\checkmark &
\njet \\
\hline
\eqref{eq:domination_by_ktilde} & \checkmark & \checkmark & \checkmark &
\checkmark & \checkmark & \checkmark & \checkmark  & \vdots &
\checkmark &  \checkmark & \njet & \checkmark & \checkmark & \checkmark
&
\checkmark \\
\hline
\eqref{eq:ka-square-domin-by-ks}& \checkmark &
\checkmark & \checkmark &
\checkmark & \checkmark & \checkmark & \checkmark  & \vdots &
\checkmark &  \checkmark & \njet & {\bf ?} & \njet & \checkmark
&
\checkmark \\
\hline
symmetry& \checkmark & \checkmark & \njet &
\njet & {\bf ?} & \njet & \njet  & \vdots &
\checkmark &  \checkmark & \njet & {\bf ?} & \njet &
\checkmark & \njet \\
\bottomrule

{} & \multicolumn{7}{c}{\bf integrable kernels} & &
\multicolumn{7}{c}{\bf non-integrable kernels}\\
\end{tabular}
\end{center}
%%%%%%%%%%%
\subsection*{Index of Conditions}\enlargethispage{1em}

\label{index}
\begin{tabbing}
Condition\phantom{s} \=\eqref{eq:integrierbarkeitsbedingung-k_s} can be found \=on p.~\pageref*{eq:integrierbarkeitsbedingung-k_s}.\\
Conditions \>\eqref{eq:domination_by_ktilde} and \eqref{eq:ka-square-domin-by-ks} \>on p.~\pageref*{eq:domination_by_ktilde}.\\
Condition \>\eqref{eq:poincare-type-ineq} \>on p.~\pageref*{eq:poincare-type-ineq}.\\
Condition \>\eqref{eq:definitheitsbedingung} \>on p.~\pageref*{eq:definitheitsbedingung}.\\
Condition \>\eqref{eq:energie-abschaetzung-nach-unten} \>on p.~\pageref*{eq:energie-abschaetzung-nach-unten}.
\end{tabbing}
%%%%%%%
\bibliographystyle{bibstyle_english_num.bst}
\bibliography{Literatur}\vspace{-1em} \enlargethispage{2em}
\end{document}

%% file: kurzkommandos-dirichlet.tex
\newcommand{\R}{  \mathbb{R}}

\newcommand{\N}{  \mathbb{N}}

\newcommand{\cA}{  \mathcal{A}}
\newcommand{\cB}{  \mathcal{B}}
\newcommand{\cC}{  \mathcal{C}}

\newcommand{\cE}{  \mathcal{E}}

\newcommand{\cI}{  \mathcal{I}}

\newcommand{\cL}{  \mathcal{L}}

\renewcommand{\epsilon}{\varepsilon}
\DeclareMathOperator{\sgn}{sgn}
\DeclareMathOperator{\supp}{supp}

\DeclareMathOperator{\dist}{dist}

\newcommand{\norm}[1]{\left\| #1 \right\|}
\newcommand{\bet}[1]{\left| #1 \right|}
\newcommand{\ska}[1]{\left\langle #1 \right\rangle}
\newcommand{\Ind}[1]{\ensuremath \mathbbm{1}_{#1}}

\renewcommand{\d}{\, \textnormal{d}}

\newcommand{\opL}{\ensuremath{\mathcal{L}}}
\newcommand{\eps}{\varepsilon}

\newcommand{\seminorm}[1]{\bigl[ #1 \bigr]}

\newcommand{\HRdk}{H(\R^d;k)}
\newcommand{\LtwoOm}{L^{2}_{\Omega}(\mathbb{R}^d)}
\newcommand{\HOmk}{H(\Omega;k)}
\newcommand{\HdualOmRd}{H_\Omega^*(\R^d;k)}
\newcommand{\HOmRd}{H_\Omega(\R^d;k)}
\newcommand{\VOmk}{V(\Omega;k)}